\documentclass[12pt]{amsart}%
\usepackage{amsfonts}
\usepackage{amsmath}
\usepackage{amssymb}
\usepackage{amsthm}
\usepackage{mathrsfs}
\usepackage{graphicx}%
\usepackage{amsmath}
\usepackage{setspace}
\usepackage{hyperref}
\usepackage{cite}
\usepackage{amsfonts}
\usepackage{algorithmic}
\usepackage{textcomp}
\usepackage{tabularx}
\usepackage{enumitem}
\usepackage{amsmath,amssymb,amsthm,mathtools}
\usepackage{mathrsfs}
\usepackage{microtype}
\usepackage{enumitem}
\usepackage{booktabs}
\usepackage{array}
\usepackage{tabularx}
\usepackage{xcolor}
\usepackage[T1]{fontenc}
\usepackage{lmodern}
\usepackage{hyperref}
\usepackage[nameinlink,capitalise,noabbrev]{cleveref}

\allowdisplaybreaks[4]
\makeatletter
\@addtoreset{equation}{section}
\makeatother

\newtheorem{theorem}{Theorem}[section]

\newtheorem{corollary}[theorem]{Corollary}

\newtheorem{lemma}[theorem]{Lemma}

\newtheorem{proposition}[theorem]{Proposition}
\newtheorem{remark}[theorem]{Remark}

\newcommand{\R}{\mathbb{R}}

\newcommand{\dd}{\,\mathrm{d}}
\newcommand{\esssup}{\mathop{\mathrm{ess\,sup}}}
\newcommand{\norm}[1]{\left\lVert #1\right\rVert}
\newcommand{\abs}[1]{\left|#1\right|}
\newcommand{\ip}[2]{\left(#1,#2\right)}
\newcommand{\Ran}{\operatorname{Ran}}
\newcommand{\dist}{\operatorname{dist}}

\newcommand{\M}{\mathcal M}

\newcommand{\MT}{\mathcal M_T}

\newcommand{\Id}{\mathrm I}

\newcommand{\one}{\mathbf 1}

\makeatletter
\def\@makefnmark{}
\makeatother

\begin{document}

\title[Boundary Spectral Inequalities]{Boundary Spectral Inequalities from Measurable Sets on $W^{2,\infty}$ manifolds and their applications}

\author{Lu Chen}
\address[Lu Chen]{Key Laboratory of Algebraic Lie Theory and Analysis of Ministry of Education, School of Mathematics and Statistics, Beijing Institute of Technology, Beijing
100081, PR China;  Tangshan Research Institute, Beijing Institute of Technology, Tangshan 063000, PR China}
\email{chenlu5818804@163.com}

\author{Haolin Liu$^{*}$}
\address[Haolin Liu]{Key Laboratory of Algebraic Lie Theory and Analysis of Ministry of Education, School of Mathematics and Statistics, Beijing Institute of Technology, Beijing
100081, PR China}
\email{haolinliu2023@126.com, 3120256210@bit.edu.cn}

\author{Hongyu Liu}
\address[Hongyu Liu]{Department of Mathematics, City University of Hong Kong, Hong Kong SAR, China}
\email{hongyu.liuip@gmail.com, hongyliu@cityu.edu.hk}

\address{}

\keywords{Boundary spectral inequality;
measurable boundary observation;
heat equation;
quantitative unique continuation.}
\thanks{$*$ Corresponding author.}
\thanks{The first author was partly supported by the  National Natural Science Foundation of China (No. 12271027) and Hebei Natural
Science Foundation (No. A2025105003).}

\begin{abstract}
Let $(\M,g)$ be a compact connected $W^{2,\infty}$ Riemannian manifold with nonempty boundary and a uniformly elliptic Lipschitz metric. We establish quantitative boundary spectral inequalities for both Dirichlet and Neumann eigenfunctions from arbitrary measurable subsets of $\partial\M\times(-s_0,s_0)$ of positive surface measure. Moreover, we apply the Dirichlet spectral inequality to establish the boundary observability estimate
\[
 \norm{u(T)}_{L^2(\M)}^2
 \le C\int_J\abs{\partial_{\nu_g}u(x,t)}^2\dd S_g\dd t,
\]
where $J\subset\partial\M\times(0,T)$ is measurable and has positive surface--time measure. A localized theorem requiring $W^{2,\infty}$ regularity only near the observed boundary patch is retained as a separate consequence. The proof of the observability inequality combines quantitative
continuation from positive-measure boundary sets with low-frequency
spectral concentration. This result shows that boundary observability of solutions to the heat equation can also be achieved in $W^{2,\infty}$ domains.
Finally, as an application, we prove that measurements of the boundary normal
derivative on $J$
uniquely determine the initial state of heat equation. Under an a priori bound in
$D(-\Delta_{g, D}^{s/2})$, the recovery satisfies a logarithmic
stability estimate of optimal order $s/2$.

\end{abstract}

\maketitle
\section{Introduction}

Spectral inequalities quantify how a low-frequency function is determined by
 observations on a smaller set.  Together with the decay of high
frequencies, they are the basis of the Lebeau--Robbiano method for heat
observability and null controllability; see
\cite{LebeauRobbiano,JerisonLebeau,Miller}. A substantial body of literature has extended observability from open observation regions to sets that are merely
measurable in space or time
\cite{ApraizEscauriaza,PhungWang,ApraizEtAl,EscauriazaMontanerZhang} or
rough manifolds \cite{BurqMoyano}.

Moreover, quantitative observability estimates for heat equations
have been further developed within the framework of spectral methods and
dissipation of high frequencies. In particular, Ervedoza and Zuazua
established sharp observability estimates for heat equations by
exploiting the Lebeau--Robbiano strategy
\cite{ErvedozaZuazua2011}.

In this paper, we study boundary spectral inequalities. At the qualitative level, the relevant boundary unique
continuation problem is the following.  Let \(u\) be harmonic
in a domain \(\Omega\), continuous up to relatively open
boundary patch \(\Sigma\subset\partial\Omega\), suppose that
\begin{equation*}
u\equiv0 \mbox{ on $\Sigma$}
\end{equation*}
and that $\partial_\nu u$
vanishes on a subset of $\Sigma$ for positive surface measure,
must $u$ vanish identically?
For $C^{1,1}(W^{2,\infty})$ domains, this qualitative boundary unique
continuation property follows from the work of Lin
\cite{Lin1991} on frequency functions and nodal sets.  More generally, Tolsa~\cite{Tolsa2023} later proved
the qualitative boundary unique continuation property for
\(C^1\) domains and, more generally, for Lipschitz graph
patches with sufficiently small slope.

These qualitative results show that positive-measure boundary
Cauchy data determine a harmonic function uniquely. However, spectral
inequality requires a quantitative statement: we need to control the solution in a neighboring interior region
in terms of the boundary data and an a priori global norm.  The quantitative propagation estimate from
positive-measure boundary subsets was established by
Burq--Zuily~\cite{BurqZuily} on
\(W^{2,\infty}\) domains, for both homogeneous
Dirichlet and homogeneous Neumann boundary conditions.

Boundary spectral inequalities lie one step beyond qualitative
unique continuation.  Apraiz, Escauriaza, Wang, and Zhang
\cite[Theorems~2, 9, and~10; Remark~11]{ApraizEtAl} established an
auxiliary-variable normal-trace inequality on an open boundary
cylinder and obtained measurable-boundary heat observability
under local real-analyticity assumptions at the observation
point.

Although the local continuation arguments are formulated in Euclidean coordinates,
they extend to the manifold setting through boundary normal coordinates and a finite
covering of $\partial \M$. In each boundary chart, the problem is reduced to a uniformly
elliptic equation on a Euclidean half-space, with the geometry encoded in the coefficients.
 $s$ denotes the auxiliary extension variable used in the spectral lifting, and
$s_0>0$ is a fixed sufficiently small parameter defining the auxiliary cylinder
$\partial \M \times (-s_0,s_0)$.

Our approach to measurable boundary sets does not use spatial
analyticity. Instead, we combine the quantitative boundary
continuation theorem of Burq--Zuily with local spectral
concentration. This gives Dirichlet and Neumann boundary spectral
inequalities on arbitrary positive-measure subsets of
$\partial\M\times(-s_0,s_0)$. For the heat equation, the same
framework yields $L^2$ observability under $W^{2,\infty}$ regularity near
the observed boundary patch, subject to the global Dirichlet
spectral hypothesis.

The results of Lin and Tolsa are qualitative and do not provide the
quantitative interpolation estimates used here. We therefore work
at the $W^{2,\infty}$ regularity level required by the quantitative
continuation theorem of Burq--Zuily and by the boundary regularity
estimates below. The main contributions are as follows.

\begin{enumerate}
\item[(i)]
We establish Dirichlet and Neumann boundary spectral inequalities
on compact connected $W^{2,\infty}$ manifolds with Lipschitz metrics.
The observation set may be any measurable subset
\[
J\subset\partial \M\times(-s_0,s_0)
\]
of positive surface measure; in particular, the observation may be
restricted to an arbitrary partial boundary portion.

\item[(ii)]
In the Neumann case, the quantitative continuation estimate involves
the full boundary gradient, whereas the observation contains only the
boundary value. We recover the tangential first-order jet on a
positive-measure subset and then recover the remaining additive
constant by one low-value point, including the zero Neumann mode.

\item[(iii)]
We prove measurable-set partial-boundary observability for the
homogeneous Dirichlet heat equation on compact $W^{2,\infty}$
manifolds. Since a fixed-time conormal spectral inequality may fail
because of cancellation between distinct eigenvalues, we recover the
time jet of the low-frequency boundary trace and transfer it to an
auxiliary elliptic extension.

\item[(iv)]
For bounded Lipschitz Euclidean domains satisfying the global
Dirichlet Lebeau--Robbiano spectral hypothesis, we localize the
Dirichlet heat argument so that $W^{2,\infty}$ boundary regularity is
required only near the observed boundary patch.

\item[(v)]
Within a fixed full lateral Dirichlet-trace class, we obtain
uniqueness and optimal conditional logarithmic stability for recovery
of the initial state from partial conormal measurements, together
with a spectral-cutoff regularization.
\end{enumerate}
The statements below are written for real-valued data.  Complex-valued data are
handled by applying the real estimates to the real and imaginary parts; the
Hilbert-space estimates are unchanged.

Let $(\M,g)$ be a compact connected $n$-dimensional $W^{2,\infty}$ manifold with nonempty boundary. In every $W^{2,\infty}$ coordinate chart the covariant metric matrix is assumed symmetric, uniformly positive definite, and Lipschitz. We write $dV_g$ and $dS_g$ for the volume and boundary measures, $\nu_g$ for the outward unit normal, and
\[
 \partial_{\nu_g}u=\ip{\nabla_gu}{\nu_g}_g
\]
for the conormal derivative. Let
\[
 A_D=-\Delta_{g,D},\qquad A_N=-\Delta_{g,N}
\]
be the nonnegative Dirichlet and Neumann Laplacians. Their $L^2(\M,dV_g)$-orthonormal eigenpairs are denoted by
\[
 A_D\phi_j=\lambda_j\phi_j,
 \qquad 0<\lambda_1\le\lambda_2\le\cdots,
\]
and
\[
 A_N\psi_j=\mu_j\psi_j,
 \qquad 0=\mu_0<\mu_1\le\mu_2\le\cdots.
\]
Fix $s_0>0$ and set
\[
 I_s=(-s_0,s_0),\qquad \Sigma_s=\partial\M\times I_s,
 \qquad d\Sigma_g=dS_g\,ds.
\]
For $\Lambda\ge1$ and finite coefficient families define
\begin{align}
 U_\Lambda(x,s)&=\sum_{\lambda_j\le\Lambda}a_j\phi_j(x)\cosh(\sqrt{\lambda_j}s),\label{eq:intro-U}\\
 W_\Lambda(x,s)&=\sum_{\mu_j\le\Lambda}b_j\psi_j(x)\cosh(\sqrt{\mu_j}s).\label{eq:intro-W}
\end{align}
Then
\begin{align}
 (\Delta_g+\partial_s^2)U_\Lambda&=0\quad\text{in }\M\times I_s,
 &U_\Lambda&=0\quad\text{on }\Sigma_s,\label{eq:intro-U-pde}\\
 (\Delta_g+\partial_s^2)W_\Lambda&=0\quad\text{in }\M\times I_s,
 &\partial_{\nu_g}W_\Lambda&=0\quad\text{on }\Sigma_s.\label{eq:intro-W-pde}
\end{align}
Throughout, $\abs J$ denotes the measure with respect to $dS_g\,ds$ or $dS_g\,dt$, according to the boundary cylinder under consideration.

\begin{theorem}
\label{thm:dir-spectral}
Let $J\subset\Sigma_s$ be measurable with $\abs J>0$. There exists $C_J>0$ such that, for every $\Lambda\ge1$ and every finite family of coefficients,
\begin{equation}
 \sum_{\lambda_j\le\Lambda}\abs{a_j}^2
 \le C_Je^{C_J\sqrt\Lambda}
 \int_J\abs{\partial_{\nu_g}U_\Lambda}^2\dd\Sigma_g.
 \label{eq:main-dir-spec}
\end{equation}
\end{theorem}

\begin{theorem}
\label{thm:neu-spectral}
Let $J\subset\Sigma_s$ be measurable with $\abs J>0$. There exists $C_J>0$ such that, for every $\Lambda\ge1$ and every finite family of coefficients,
\begin{equation}
 \sum_{\mu_j\le\Lambda}\abs{b_j}^2
 \le C_Je^{C_J\sqrt\Lambda}
 \int_J\abs{W_\Lambda}^2\dd\Sigma_g.
 \label{eq:main-neu-spec}
\end{equation}
\end{theorem}

The constants may depend on the geometry and location of $J$. The product version needed for the parabolic argument is stated in Corollary~\ref{cor:uniform-product}. These estimates are quantitative versions of boundary unique continuation from positive-measure sets.

Consider the homogeneous Dirichlet heat equation
\begin{equation}
 \begin{cases}
  \partial_tu-\Delta_gu=0&\text{in }\M\times(0,T),\\
  u=0&\text{on }\partial\M\times(0,T),\\
  u(\cdot,0)=u_0\in L^2(\M).
 \end{cases}
 \label{eq:heat-intro}
\end{equation}

\begin{theorem}
\label{thm:heat-observability}
Let $J\subset\partial\M\times(0,T)$ be measurable with $\abs J>0$. There exists $C=C(\M,g,T,J)>0$ such that every solution of \eqref{eq:heat-intro} satisfies
\begin{equation}
 \norm{u(T)}_{L^2(\M)}^2
 \le C\int_J\abs{\partial_{\nu_g}u(x,t)}^2\dd S_g\dd t.
 \label{eq:main-obs}
\end{equation}

\end{theorem}
For every $t>0$, analytic-semigroup smoothing gives $u(t)\in D(A_D^m)$ for all $m\ge1$, and hence the conormal trace is well defined. If the observation integral is infinite, the conclusion is understood in the extended sense.

Moreover, for a product set $J=\Gamma\times E$, with $\Gamma$ fixed and
$E\subset(0,T)$ measurable, Theorem~\ref{thm:heat-observability} is compatible with the
abstract measurable-time framework of Wang--Zhang~\cite{WangZh}.

The global $W^{2,\infty}$ hypothesis is convenient for the unified Dirichlet--Neumann theory, but it is stronger than what is required by the Dirichlet heat argument on a Euclidean domain. In Section~\ref{sec:localized}, we prove the following localized result.

\begin{theorem}
\label{thm:localized}
Let $(\M,g)$ be a compact connected Lipschitz Riemannian manifold with nonempty boundary, equipped with a uniformly elliptic Lipschitz metric, and assume that its Dirichlet Laplacian $A_D=-\Delta_{g,D}$ satisfies the intrinsic Dirichlet Lebeau--Robbiano hypothesis \eqref{eq:local-LR}. Let $\Gamma_0,\Gamma_1\subset\partial\M$ be relatively open with $\Gamma_0\Subset\Gamma_1$, and assume that the boundary is of class $W^{2,\infty}$ in a neighborhood of $\Gamma_1$. If $J\subset\Gamma_0\times(0,T)$ is measurable and $\abs J>0$, then every homogeneous Dirichlet heat solution satisfies
\begin{equation}
 \norm{u(T)}_{L^2(\M)}^2
 \le C_{\M,g,T,J,\Gamma_0,\Gamma_1}
 \int_J\abs{\partial_{\nu_g}u}^2\dd S_g\dd t.
 \label{eq:localized-main}
\end{equation}

\end{theorem}
For positive times, the conormal trace on $\Gamma_0$ is defined by the local $W^{2,\infty}$ boundary regularity. As in Theorem~\ref{thm:heat-observability}, the right-hand side is understood as an extended nonnegative value, and the proof uses only a time block compactly contained in $(0,T)$.\\

The hypotheses of Theorem~\ref{thm:localized} are the natural manifold generalization of the assumptions underlying the Euclidean measurable-boundary result of Apraiz--Escauriaza--Wang--Zhang~\cite[Theorem~2 and Remark~11]{ApraizEtAl}. In the Euclidean case $\M=\Omega\subset\R^n$, their interior-ball spectral hypothesis \cite[(1.4)]{ApraizEtAl} implies the intrinsic condition \eqref{eq:local-LR} used here, while their boundary argument assumes local real analyticity on the observed patch. Thus the present theorem keeps the same global spectral mechanism, but formulates it intrinsically on a Lipschitz manifold and requires only local $W^{2,\infty}$ regularity near the observation set. This should not be read as a strict strengthening of \cite[Theorem~2]{ApraizEtAl}: whenever both results apply, their $L^1$ observation is stronger than the $L^2$ observation in \eqref{eq:localized-main}.

Next we consider the application of our results in inverse problem. The recovery of an unknown initial temperature from lateral
Cauchy data is a foundational inverse problem for parabolic
equations. Classical approaches are based predominantly on
Carleman estimates and related conditional-stability arguments;
see, for example,
\cite{Klibanov2006,Yamamoto2009,
LiYamamotoZou2009,ImanuvilovYamamoto2014}.
Regularized reconstruction has also been developed through
quasi-reversibility, Tikhonov-type minimization, and variational
methods; see
\cite{LattesLions1967,LiYamamotoZou2009,
DelSantoPrizzi2025}.
We also refer to
\cite{BenabdallahGaitanLeRousseau2007}
for the simultaneous recovery of a discontinuous diffusion
coefficient and an initial condition from partial boundary and
positive-time measurements.

Our approach differs from the standard parabolic Carleman-estimate method at the observability stage. We do not derive the inverse result from a parabolic
Carleman estimate. Instead, we introduce a
boundary spectral--time-jet framework. Its starting point
is our quantitative boundary spectral inequality on an arbitrary
measurable set of positive surface measure.

As a consequence, the full lateral Dirichlet trace together with
the Neumann trace on an arbitrary measurable boundary
space--time subset of positive measure uniquely determines the
initial state. Under a natural spectral a priori condition, we
further obtain sharp conditional logarithmic stability, an
all-noise spectral-cutoff reconstruction, the failure of every
H\"older stability estimate, and the optimality of the
logarithmic power.

We next formulate the inverse problem in the full-Dirichlet/partial-conormal form. Put
\[
 \MT=\M\times(0,T),\qquad \Sigma_T=\partial\M\times(0,T).
\]
Let $\mathscr H_T$ denote the class of functions
\[
 u\in C([0,T];L^2(\M))\cap H^1_{\mathrm{loc}}((0,T];L^2(\M))
 \cap L^2_{\mathrm{loc}}((0,T];H^2(\M))
\]
which solve $\partial_tu-\Delta_gu=0$ in $\MT$. For every $\varepsilon\in(0,T)$, the spatial trace theorem gives
\[
 u|_{\partial\M\times(\varepsilon,T)}
 \in L^2\bigl((\varepsilon,T);H^{3/2}(\partial\M)\bigr),
 \qquad
 \partial_{\nu_g}u
 \in L^2\bigl((\varepsilon,T);H^{1/2}(\partial\M)\bigr).
\]
We regard the full lateral Dirichlet trace and the conormal trace as locally defined positive-time traces; equalities of such traces are understood on every strip $\partial\M\times(\varepsilon,T)$. For $u\in\mathscr H_T$, set $f=u(\cdot,0)$ and define the boundary observation dataset
\begin{equation}
 \mathcal M_J^T(u)
 :=\left(
  u|_{\Sigma_T},\,
  \partial_{\nu_g}u|_J
 \right),
 \qquad
 J\subset\partial\M\times(0,T),\quad \abs J>0.
 \label{eq:dataset}
\end{equation}
The inverse problem is to determine the unknown initial state $f$ from this dataset.

\begin{theorem}
\label{thm:inverse-unique}
Let $J\subset\partial\M\times(0,T)$ be measurable with $\abs J>0$. If $u_1,u_2\in\mathscr H_T$ and
\begin{equation}
 \mathcal M_J^T(u_1)=\mathcal M_J^T(u_2),
 \label{eq:dataset-equality}
\end{equation}
then
\[
 u_1(\cdot,0)=u_2(\cdot,0)\quad\text{in }\M.
\]
\end{theorem}
Conversely, if the two solutions have the same initial state and the same full Dirichlet trace, then the forward uniqueness theorem gives $u_1=u_2$, and hence their datasets agree. Equivalently, on every fixed full Dirichlet-trace class the partial conormal trace on $J$ is one-to-one with respect to the initial state. Moreover, if the full Dirichlet traces agree, then
\begin{equation}
 \norm{u_1(T)-u_2(T)}_{L^2(\M)}^2
 \le C\int_J\abs{\partial_{\nu_g}(u_1-u_2)}^2\dd S_g\dd t.
 \label{eq:terminal-stability}
\end{equation}
For stability and noisy-data reconstruction, choose $\tau\in(0,T)$ such that
\[
 J_\tau=J\cap\bigl(\partial\M\times(\tau,T)\bigr)
\]
has positive measure. Such a choice always exists.

\begin{theorem}
\label{thm:inverse-log}
Let $u_i\in\mathscr H_T$ have the same full lateral Dirichlet trace, set
\[
 f_i=u_i(\cdot,0),\qquad h=f_1-f_2,
\]
and suppose that, for some $s>0$,
\[
 h\in D(A_D^{s/2}),\qquad
 \norm{A_D^{s/2}h}_{L^2(\M)}\le \mathfrak M.
\]
Then
\begin{equation}
 \norm{f_1-f_2}_{L^2(\M)}
 \le C\mathfrak M
 \left[
  \log\left(
   e+\frac{c\mathfrak M}
   {\norm{\partial_{\nu_g}(u_1-u_2)}_{L^2(J_\tau)}}
  \right)
 \right]^{-s/2}.
 \label{eq:main-log}
\end{equation}
The constants depend only on $s,T,\M,g$, and $J_\tau$.
\end{theorem}
If the data difference vanishes, the right-hand side is interpreted as zero. On the same a priori ball, no H\"older stability estimate can hold. Furthermore, for observations separated from $t=0$, no estimate of the form \eqref{eq:main-log} with the exponent $s/2$ replaced by any larger exponent can hold uniformly. Thus the logarithmic character and the power $s/2$ are sharp in the spectral scale $D(A_D^{s/2})$. The same logarithmic order is achieved by the all-noise spectral-cutoff reconstruction constructed in Section~\ref{sec:stability}, either for homogeneous Dirichlet data or, more generally, relative to a fixed reference solution in the prescribed full Dirichlet-trace class.

\section{Proof of Theorem~\ref{thm:dir-spectral}: the Dirichlet boundary spectral inequality}
\label{sec:dirichlet}

In this section we prove the Dirichlet boundary spectral inequality. The main inputs are the measurable-boundary continuation estimate of Lemma~\ref{lem:BZ} and the local spectral concentration estimates of Lemmas~\ref{lem:BM-open}--\ref{lem:BM-product}. The capping, reflection, energy, and boundary regularity results in Lemmas~\ref{lem:cap}--\ref{lem:regularity} provide the geometric and a priori estimates required to apply those inputs.

Let $d=n+1$. In a local coordinate chart, if $Q$ is a boundary point of a $W^{2,\infty}$ domain $\mathcal O\subset\R^d$, set
\[
 D_r(Q)=B_r(Q)\cap\mathcal O,
 \qquad
 \Gamma_r(Q)=B_r(Q)\cap\partial\mathcal O.
\]
For a positive Lipschitz density $\kappa$ and a Lipschitz Riemannian metric $G=(G_{jk})$, write
\[
 \Delta_Gv
 =\frac1\kappa\partial_j\bigl(\kappa G^{jk}\partial_kv\bigr).
\]

\begin{lemma}
\label{lem:BZ}
Let $\mathcal O\subset\R^d$ have $W^{2,\infty}$ boundary and let $m_0>0$. There exists $r_0>0$ such that, for every $Q\in\partial\mathcal O$ and $0<r<r_0$, there are $C=C(r,m_0)>0$ and $\alpha=\alpha(r,m_0)\in(0,1)$ with the following properties.

\begin{enumerate}[label=\textnormal{(\roman*)}]
\item If $v\in W^{1,\infty}(\mathcal O)$, $\Delta_Gv=0$, and $v=0$ on $\Gamma_r(Q)$, then for every measurable $E\subset\Gamma_r(Q)$ satisfying $\abs E_{d-1}\ge m_0$,
\begin{equation}
 \sup_{D_{r/2}(Q)}\bigl(\abs v+\abs{\nabla v}\bigr)
 \le C
 \left(\esssup_E\abs{\partial_{\nu_G}v}\right)^\alpha
 \left(\sup_{D_r(Q)}\abs{\nabla v}\right)^{1-\alpha}.
 \label{eq:BZ-D}
\end{equation}
\item If $v\in W^{1,\infty}(\mathcal O)$, $\Delta_Gv=0$, and $\partial_{\nu_G}v=0$ on $\Gamma_r(Q)$, then for every measurable $E\subset\Gamma_r(Q)$ satisfying $\abs E_{d-1}\ge m_0$,
\begin{equation}
 \sup_{D_{r/2}(Q)}\abs{\nabla v}
 \le C
 \left(\esssup_E\abs{\nabla v}\right)^\alpha
 \left(\sup_{D_r(Q)}\abs{\nabla v}\right)^{1-\alpha}.
 \label{eq:BZ-N}
\end{equation}
\end{enumerate}
The constants are uniform over all sets $E$ satisfying the prescribed measure lower bound.
\end{lemma}

\begin{proof}
This is the local form of \cite[Theorems~2.1--2.2]{BurqZuily}. The theorem already permits Lipschitz metrics and positive Lipschitz densities. In every application below, the product Laplace--Beltrami operator is first written in local divergence form, so it falls directly into this framework.
\end{proof}

\begin{lemma}
\label{lem:cap}
Let $Q_0=(x_0,s_*)\in\partial\M\times I_s$ and suppose that $\dist(s_*,\{\pm s_0\})>0$. There exists $r_*>0$ such that, for every $0<r<r_*$, one can find a bounded $W^{2,\infty}$ domain
\[
 \mathcal O_{Q_0,r}\subset\M\times I_s
\]
and a bi-Lipschitz $W^{2,\infty}$ boundary coordinate map $\Phi$, defined on a fixed multiple of
\[
 B_r^+=B_r\cap\{y_d>0\},
\]
with the following properties.
\begin{enumerate}[label=\textnormal{(\roman*)}]
\item In $B_{3r}(Q_0)$, the domain $\mathcal O_{Q_0,r}$ coincides with $\M\times I_s$ and its boundary coincides with the lateral boundary $\partial\M\times I_s$.
\item The image $\Phi(B_{4r}\cap\{y_d=0\})$ is the corresponding lateral boundary patch.
\item If $z\in\Phi(B_{r/4}\cap\{y_d=0\})$ and $y\in\Phi(B_{r/4}^+)$, then there exists a rectifiable curve
\[
 \gamma\subset\Phi(B_{r/2}^+)
\]
joining $z$ to $y$ and satisfying $\operatorname{length}(\gamma)\le Cr$.
\end{enumerate}
The rescaled $W^{2,\infty}$ character is uniform and depends only on a fixed finite boundary atlas, the ellipticity and Lipschitz bounds of $g$, and the distance from $Q_0$ to the artificial end faces.
\end{lemma}

\begin{proof}
We first construct the boundary coordinate map and then build the cap in those same coordinates. Work in a spatial $W^{2,\infty}$ chart in which
\[
 \M\cap U=\{X=(X',X_n):r(X):=X_n-\varphi(X')>0\},
 \qquad \varphi\in W^{2,\infty}.
\]
Translate the tangential variables so that $x_0=X_0(0)$, and use the local auxiliary coordinate $s-s_*$. For notational simplicity we again denote this translated auxiliary coordinate by $s$. Write
\[
 X_0(x')=(x',\varphi(x')),
 \qquad
 \eta(x')=Dr(X_0(x'))=(-\nabla\varphi(x'),1).
\]
Here $\eta$ is the boundary conormal covector. If $g(X)$ denotes the covariant metric matrix, define the non-normalized metric normal
\[
 N(x')=g(X_0(x'))^{-1}\eta(x').
\]
The field $N$ is Lipschitz and uniformly transverse to the boundary because
\begin{equation}
 \eta(x')\cdot N(x')
 =\eta(x')^Tg(X_0(x'))^{-1}\eta(x')\ge c_0>0.
 \label{eq:metric-transversality}
\end{equation}
Extend $N$ Lipschitzly to a slightly larger tangential neighborhood and set $N_t=\rho_t*N$, where convolution is taken only in $x'$. For $t>0$, define
\begin{equation}
 \Phi(x',s,t)=\bigl(X_0(x')+tN_t(x'),s_*+s\bigr).
 \label{eq:metric-normal-map}
\end{equation}
The standard cancellation estimates for mollification of a Lipschitz field give
\begin{align*}
 &\norm{D_{x'}N_t}_{L^\infty}+\norm{\partial_tN_t}_{L^\infty}\\
 &\qquad
 +t\left(
  \norm{D_{x'}^2N_t}_{L^\infty}
  +\norm{D_{x'}\partial_tN_t}_{L^\infty}
  +\norm{\partial_t^2N_t}_{L^\infty}
 \right)
 \le C\norm N_{W^{1,\infty}}.
\end{align*}
Moreover, $N_t\to N$ and $t\partial_tN_t\to0$ uniformly as $t\downarrow0$. Hence $tN_t$ extends as a $W^{2,\infty}$ function to $t=0$ and
\[
 \partial_t\Phi(x',s,0)=(N(x'),0).
\]
At $t=0$ the tangential coordinate vectors are
\[
 T_\alpha=\partial_{x_\alpha}\Phi=(e_\alpha,\partial_\alpha\varphi,0),
 \qquad
 T_s=\partial_s\Phi=(0,\ldots,0,1).
\]
For the product metric $G=g\oplus ds^2$,
\[
 G(T_\alpha,\partial_t\Phi)=T_\alpha^TgN=T_\alpha^T\eta=0,
 \qquad
 G(T_s,\partial_t\Phi)=0.
\]
Thus the last coordinate vector is $G$-orthogonal to all tangential coordinate vectors on the flattened boundary. Also,
\[
 \partial_t(r\circ\Phi)(x',s,0)=\eta(x')\cdot N(x')\ge c_0.
\]
After shrinking the chart, $D\Phi$ is uniformly invertible, $\partial_t(r\circ\Phi)\ge c_0/2$, and $\Phi$ maps an upper half-neighborhood bi-Lipschitzly onto the corresponding portion of $\M\times I_s$. The quantitative inverse function theorem for $W^{2,\infty}$ maps gives
\[
 \Phi,\Phi^{-1}\in W^{2,\infty}.
\]

We now construct the cap in the $y=(x',s,t)$ coordinates supplied by this same map. Choose $\chi\in C^\infty([0,\infty))$ with $0\le\chi\le1$, $\chi=1$ on $[0,3]$, and $\chi=0$ on $[4,\infty)$, and put $a=1-\chi$. Fix a sufficiently large dimensionless constant $L_*>1$ and define
\begin{equation}
 \begin{aligned}
 F_r(y)&=L_*r y_d-a(\abs y/r)\abs y^2,\\
 \widehat{\mathcal O}_r&=\text{the connected component of }\{F_r>0\}\\
 &\qquad\text{which meets }B_{3r}^+.
 \end{aligned}
 \label{eq:cap-def}
\end{equation}
If $\abs y<3r$, then $F_r=L_*r y_d$, so
\[
 \widehat{\mathcal O}_r\cap B_{3r}=B_{3r}^+.
\]
If $y_d\le0$, then $F_r(y)\le0$, hence $\widehat{\mathcal O}_r\subset\{y_d>0\}$. Outside $B_{4r}$ one has $a=1$, and $F_r>0$ is equivalent to
\[
 \abs{y'}^2+\left(y_d-\frac{L_*r}{2}\right)^2
 <\left(\frac{L_*r}{2}\right)^2,
\]
so the set is bounded.

It remains to check that zero is a regular value of $F_r$. On the annulus $3r\le\abs y\le4r$, the identity $F_r=0$ gives
\[
 0\le y_d\le\frac{16r}{L_*}.
\]
Moreover,
\[
 \partial_{y_d}F_r
 =L_*r-
 \left[a'(\abs y/r)\frac{\abs y}{r}+2a(\abs y/r)\right]y_d.
\]
Choosing $L_*$ larger than a fixed constant depending only on $\chi$ gives
\[
 \partial_{y_d}F_r\ge\frac{L_*r}{2}
\]
throughout the transition annulus. In $B_{3r}$ the gradient equals $L_*r e_d$, while outside $B_{4r}$ the level set is the sphere displayed above and has nonvanishing gradient. Thus $\widehat{\mathcal O}_r$ is a bounded smooth domain, contained in the upper half-space and agreeing with it in $B_{3r}$. Since
\[
 F_r(y)=r^2F_1(y/r),
\]
its rescaled geometry is uniform.

The diameter of $\widehat{\mathcal O}_r$ is bounded by a fixed multiple of $r$ depending only on $L_*$. Choose $r_*>0$ so small that, for $0<r<r_*$, the map $\Phi$ is defined on a fixed ball containing $\overline{\widehat{\mathcal O}_r}$ and the image remains separated from $s=\pm s_0$. Define
\[
 \mathcal O_{Q_0,r}=\Phi(\widehat{\mathcal O}_r).
\]
Because $\widehat{\mathcal O}_r=B_{3r}^+$ in $B_{3r}$ and $\Phi$ parametrizes $\M\times I_s$ on the upper side of $\{y_d=0\}$, property (i) follows in the corresponding coordinate neighborhood. Property (ii) is part of the construction of $\Phi$. For (iii), in the flattened variables join the boundary point first to a point at height comparable to $r$ and then to the interior point by a straight segment. After reducing the fixed fractions in the statement, if necessary, the polygonal path lies in $B_{r/2}^+$ and has length at most $Cr$; its image under the uniformly bi-Lipschitz map $\Phi$ has the same properties up to a fixed constant. This proves the lemma.
\end{proof}

In the coordinates of Lemma~\ref{lem:cap},  set
\[
\widetilde\kappa(y)
=
|\det J_\Phi(y)|\,\kappa(\Phi(y)),
\qquad
\widetilde G(y)
=
J_\Phi(y)^T G(\Phi(y))J_\Phi(y).
\]write the local form of the product operator as
\[
 \frac1\kappa\partial_i(A^{ij}\partial_jv)=0.
\]
With $J_\Phi=D\Phi$, the full weak coefficient matrix is
\begin{equation}
 A(y)
=
\widetilde\kappa(y)\,
\widetilde G(y)^{-1}
=
|\det J_\Phi(y)|\,\kappa(\Phi(y))
J_\Phi(y)^{-1}G(\Phi(y))^{-1}J_\Phi(y)^{-T}.
 \label{eq:transformed-A}
\end{equation}
It is symmetric, uniformly elliptic, and Lipschitz. Since $J_\Phi^TGJ_\Phi$ is block diagonal on $\{y_d=0\}$ and
\[
 J_\Phi^{-1}G^{-1}J_\Phi^{-T}=(J_\Phi^TGJ_\Phi)^{-1},
\]
the tangential--normal blocks of $A$ vanish on the reflecting hyperplane. Writing
\[
 A=\begin{pmatrix}A_{\tau\tau}&a_{\tau d}\\a_{d\tau}&a_{dd}\end{pmatrix},
 \qquad
 S=\operatorname{diag}(1,\ldots,1,-1),
\]
define the reflected coefficient matrix by
\begin{equation}
 A^e(y',-y_d)=SA(y',y_d)S,
 \qquad y_d>0.
 \label{eq:reflection-matrix}
\end{equation}
The diagonal blocks are extended evenly and the mixed blocks oddly. Because
\[
 a_{\tau d}=a_{d\tau}=0\quad\text{on }\{y_d=0\},
\]
the reflected matrix $A^e$ is Lipschitz across the reflecting hyperplane. Uniform ellipticity is preserved because $S$ is orthogonal.

We shall apply the boundary continuation theorem to the pullback metric $\widetilde G=\Phi^*G$ and the corresponding pullback density. With this convention the normalized geometric conormal derivative is coordinate invariant:
\begin{equation}
 \partial_{\nu_{\widetilde G}}(V\circ\Phi)
 =\bigl(\partial_{\nu_G}V\bigr)\circ\Phi
 \quad\text{on }\{y_d=0\}.
 \label{eq:conormal-invariance}
\end{equation}
If instead the transformed equation is written only through the matrix $A$ in \eqref{eq:transformed-A}, its unnormalized boundary flux $(A\nabla(V\circ\Phi))\cdot e_d$ differs from the left-hand side of \eqref{eq:conormal-invariance} by a positive Lipschitz factor bounded above and below by constants depending only on the fixed chart. The pullback of $dS_g\,ds$ is likewise equivalent to the coordinate surface measure. Thus all boundary essential-supremum and $L^2$ estimates below are unchanged after renaming geometric constants.

\begin{lemma}[Dirichlet odd reflection]
\label{lem:odd-reflection}
Suppose $V$ satisfies
\[
 (\Delta_g+\partial_s^2)V=0
\]
near a lateral boundary cylinder and $V=0$ on the lateral boundary. In the metric-normal coordinates of Lemma~\ref{lem:cap}, the odd reflection of $v=V\circ\Phi$ is an $H^1$ weak solution of
\[
 \operatorname{div}(A^e\nabla v^e)=0
\]
in a full ball, with $A^e$ given by \eqref{eq:reflection-matrix}. After the boundary $C^{1,\sigma}$ regularity of Lemma~\ref{lem:regularity} has been established,
\begin{equation}
 \abs{\nabla v^e}\asymp\abs{\partial_{\nu_g}V}
 \quad\text{on the reflecting hyperplane}.
 \label{eq:Dir-reflection-gradient}
\end{equation}
\end{lemma}

\begin{proof}
Set $v^e(y',y_d)=v(y',y_d)$ for $y_d>0$ and
\[
 v^e(y',y_d)=-v(y',-y_d)
\]
for $y_d<0$. The zero trace implies $v^e\in H^1$ across the hyperplane. Let $\zeta\in C_c^\infty$ be a test function in the full ball. Splitting the integral into the upper and lower half-balls, changing variables in the lower half, and using the coefficient extension \eqref{eq:reflection-matrix} reduces the weak formulation to the upper half-ball with the test function
\[
 \zeta(y',y_d)-\zeta(y',-y_d),
\]
which has zero trace on $y_d=0$. The weak Dirichlet equation therefore gives zero. No pointwise boundary differentiability is used at this stage. Once a $C^{1,\sigma}$ representative is known, the zero boundary trace implies that all tangential derivatives vanish, while the metric-normal coordinate direction is uniformly equivalent to the geometric normal. This gives \eqref{eq:Dir-reflection-gradient}.
\end{proof}

\begin{lemma}[Neumann even reflection]
\label{lem:even-reflection}
Suppose $W$ satisfies
\[
 (\Delta_g+\partial_s^2)W=0,
 \qquad
 \partial_{\nu_g}W=0
\]
on a lateral boundary cylinder. In the same coordinates, the even reflection of $w=W\circ\Phi$ is an $H^1$ weak solution of
\[
 \operatorname{div}(A^e\nabla w^e)=0
\]
in a full ball. After Lemma~\ref{lem:regularity},
\begin{equation}
 \abs{\nabla_GW}
 \asymp
 \abs{\nabla_\tau(W|_{\partial\M\times I_s})}
 \quad\text{on the lateral boundary}.
 \label{eq:Neu-reflection-gradient}
\end{equation}
\end{lemma}

\begin{proof}
Set $w^e(y',y_d)=w(y',\abs{y_d})$ and use the same reflected coefficient matrix. After splitting the weak formulation and changing variables, the two conormal fluxes on the reflecting hyperplane cancel because the original conormal derivative is zero. Thus $w^e$ is a weak solution in the full ball. Once the $C^{1,\sigma}$ representative is available, the normal derivative vanishes pointwise and the full boundary gradient consists precisely of the tangential derivatives. This proves \eqref{eq:Neu-reflection-gradient}.
\end{proof}

\begin{lemma}
\label{lem:BM-open}
Fix a nonempty open set $\omega\Subset\M$. In both the Dirichlet and Neumann cases, there exists $C_\omega>0$ such that, for every $\Lambda\ge1$ and
\[
 f_\Lambda=\sum_{\kappa_j\le\Lambda}c_je_j,
\]
where $\kappa_j$ denotes either $\lambda_j$ or $\mu_j$, one has
\begin{equation}
 \sum_{\kappa_j\le\Lambda}\abs{c_j}^2
 \le C_\omega e^{C_\omega\sqrt\Lambda}
 \norm{f_\Lambda}_{L^2(\omega)}^2.
 \label{eq:BM-open}
\end{equation}
\end{lemma}

\begin{proof}
Burq--Moyano \cite[Theorem~1, estimate~(1.6), and Section~3]{BurqMoyano} write the eigenvalues as $\rho_j^2$ and prove, on compact $W^{2,\infty}$ manifolds with boundary and Lipschitz metric, for either Dirichlet or Neumann conditions, that for every fixed positive-measure set $E_1$,
\[
 \norm{\sum_{\rho_j\le R}c_je_j}_{L^\infty(\M)}
 \le Ce^{CR}
 \norm{\sum_{\rho_j\le R}c_je_j}_{L^1(E_1)}.
\]
Take $E_1=\omega$, $R=\sqrt\Lambda$, and $\rho_j^2=\kappa_j$. Orthogonality gives
\[
 \norm{f_\Lambda}_{L^2(\M)}^2
 =\sum_{\kappa_j\le\Lambda}\abs{c_j}^2.
\]
Together with
\[
 \norm f_{L^2(\M)}\le \operatorname{Vol}_g(\M)^{1/2}\norm f_{L^\infty(\M)},
 \qquad
 \norm f_{L^1(\omega)}\le \operatorname{Vol}_g(\omega)^{1/2}\norm f_{L^2(\omega)},
\]
this yields \eqref{eq:BM-open} after squaring and renaming the constant.
\end{proof}

\begin{lemma}
\label{lem:BM-product}
Let $V_\Lambda$ denote either $U_\Lambda$ or $W_\Lambda$, and write
\[
 A_\Lambda^2=\sum_{\kappa_j\le\Lambda}\abs{c_j}^2.
\]
Fix $\omega\Subset\M$ and a nondegenerate interval $I_1\Subset I_s$. Then
\begin{equation}
 A_\Lambda^2
 \le Ce^{C\sqrt\Lambda}
 \norm{V_\Lambda}_{L^2(\omega\times I_1)}^2.
 \label{eq:BM-product}
\end{equation}
\end{lemma}

\begin{proof}
For each fixed $s\in I_1$,
\[
 V_\Lambda(\cdot,s)
 =\sum_{\kappa_j\le\Lambda}c_j\cosh(\sqrt{\kappa_j}s)e_j
\]
is again a low-frequency sum. Applying Lemma~\ref{lem:BM-open}, integrating in $s$, and using $\cosh^2(\sqrt{\kappa_j}s)\ge1$ give
\[
 \abs{I_1}A_\Lambda^2
 \le Ce^{C\sqrt\Lambda}
 \norm{V_\Lambda}_{L^2(\omega\times I_1)}^2.
\]
Absorb $\abs{I_1}^{-1}$ into the constant.
\end{proof}

\begin{lemma}
\label{lem:energy}
For both the Dirichlet and Neumann extensions,
\begin{align}
 \norm{V_\Lambda}_{L^2(\M\times I_s)}
 &\le Ce^{s_0\sqrt\Lambda}A_\Lambda,
 \label{eq:energy-L2}\\
 \norm{V_\Lambda}_{H^1(\M\times I_s)}
 &\le C(1+\sqrt\Lambda)e^{s_0\sqrt\Lambda}A_\Lambda.
 \label{eq:energy-H1}
\end{align}
\end{lemma}

\begin{proof}
Orthogonality implies, for every $s\in I_s$,
\[
 \norm{V_\Lambda(\cdot,s)}_{L^2(\M)}^2
 =\sum_{\kappa_j\le\Lambda}\abs{c_j}^2\cosh^2(\sqrt{\kappa_j}s)
 \le e^{2s_0\sqrt\Lambda}A_\Lambda^2.
\]
Integration in $s$ proves \eqref{eq:energy-L2}. The form identity gives
\[
 \norm{\nabla_ge_j}_{L^2(\M)}^2=\kappa_j
\]
for both boundary conditions, and
\[
 \abs{\partial_s\cosh(\sqrt{\kappa_j}s)}
 \le\sqrt\Lambda e^{s_0\sqrt\Lambda}.
\]
These estimates yield \eqref{eq:energy-H1}.
\end{proof}

\begin{lemma}
\label{lem:regularity}
Let $K$ be a fixed compact subset of $\overline\M\times I_s$ separated from the artificial end faces $s=\pm s_0$. For every fixed $0<\sigma<1$, there exists $C_K>0$ such that
\begin{equation}
 \norm{V_\Lambda}_{C^{1,\sigma}(K)}
 \le C_Ke^{C_K\sqrt\Lambda}A_\Lambda.
 \label{eq:C1sigma}
\end{equation}
\end{lemma}

\begin{proof}
Cover $K$ by finitely many interior balls and lateral boundary half-balls whose doubled radii remain separated from $s=\pm s_0$. Interior balls are handled by the standard interior estimate for harmonic functions on a Lipschitz metric background.

On a boundary half-ball, flatten the boundary by the map of Lemma~\ref{lem:cap} and use Lemmas~\ref{lem:odd-reflection} and \ref{lem:even-reflection}. The reflected function $\widetilde V$ is a weak solution of
\[
 \partial_i(a^e_{ij}\partial_j\widetilde V)=0
 \quad\text{in }B_r,
 \qquad
 A^e=(a^e_{ij})\in W^{1,\infty}(B_r),
\]
with uniform ellipticity and Lipschitz bounds. Since $A^e\in W^{1,\infty}$, for every fixed $0<\sigma<1$ it belongs to $C^{0,\sigma}$ with
\[
 [A^e]_{C^{0,\sigma}(B_r)}
 \le (2r)^{1-\sigma}\norm{\nabla A^e}_{L^\infty(B_r)}.
\]
The interior gradient Schauder estimate for weak solutions of divergence-form equations with H\"older coefficients, combined with local boundedness, gives on concentric balls
\begin{align}
 &\norm{\widetilde V}_{L^\infty(B_{r/2})}
 +r\norm{\nabla\widetilde V}_{L^\infty(B_{r/2})}
 +r^{1+\sigma}[\nabla\widetilde V]_{C^{0,\sigma}(B_{r/2})}
 \notag\\
 &\qquad\le Cr^{-d/2}\norm{\widetilde V}_{L^2(B_r)}.
 \label{eq:Schauder-local}
\end{align}
The constant depends only on the dimension, ellipticity, $\sigma$, and the scale-invariant H\"older norm of the coefficients; see \cite[Chapter~8]{GilbargTrudinger} and \cite{LadyzhenskayaUraltseva}. The reflected $L^2$ norm is bounded by a fixed multiple of the norm on the original half-ball. Since $\Phi$ and $\Phi^{-1}$ are $W^{2,\infty}$ with uniform bounds, composition with the coordinate map preserves $C^{1,\sigma}$.

Summing over the fixed finite cover gives
\[
 \norm{V_\Lambda}_{C^{1,\sigma}(K)}
 \le C_K\norm{V_\Lambda}_{L^2(\M\times I_s)}.
\]
Inserting Lemma~\ref{lem:energy} proves \eqref{eq:C1sigma}. Notice that the weak reflection is established before any pointwise boundary differentiability is used; the pointwise identities in Lemmas~\ref{lem:odd-reflection}--\ref{lem:even-reflection} are interpreted only after the present regularity estimate is obtained.
\end{proof}

\begin{lemma}
\label{lem:density-cylinder}
Let $J\subset\Sigma_s$ be measurable with $\abs J>0$. Then there exists a boundary coordinate cylinder
\[
 Q=\Delta\times I_0\Subset\Sigma_s,
\]
separated from the artificial end faces, and $S=J\cap Q$ such that
\begin{equation}
 \abs S\ge\vartheta\abs Q>0.
 \label{eq:density-cylinder}
\end{equation}
\end{lemma}

\begin{proof}
Almost every point of $J$ is a density point. The end faces have zero $n$-dimensional boundary measure, so one may choose a density point $(x_0,s_*)$ with $\abs{s_*}<s_0$. In a $W^{2,\infty}$ boundary chart, $dS_g\,ds$ is equivalent to Lebesgue measure in $\R^n$, and coordinate cylinders form a regular differentiation basis. The density ratio therefore tends to one as the cylinder shrinks. Choose one compactly contained in $\Sigma_s$ for which the ratio is at least $1/2$.
\end{proof}

\begin{lemma}
\label{lem:low-value}
Let $S$ have finite positive measure and $h\in L^2(S)$. Define
\[
 S_h=\left\{z\in S:\abs{h(z)}^2\le\frac2{\abs S}\norm h_{L^2(S)}^2\right\}.
\]
Then $\abs{S_h}\ge\abs S/2$ and
\begin{equation}
 \esssup_{S_h}\abs h
 \le\left(\frac2{\abs S}\right)^{1/2}\norm h_{L^2(S)}.
 \label{eq:low-value}
\end{equation}
\end{lemma}

\begin{proof}
If $\abs{S\setminus S_h}>\abs S/2$, integration over the complement would give a strict lower bound larger than $\norm h_{L^2(S)}^2$, a contradiction.
\end{proof}

\begin{proof}[Proof of Theorem~\ref{thm:dir-spectral}]
Set
\[
 A_\Lambda=\left(\sum_{\lambda_j\le\Lambda}\abs{a_j}^2\right)^{1/2},
 \qquad
 O_J=\norm{\partial_{\nu_g}U_\Lambda}_{L^2(J)}.
\]
The assertion is immediate if $A_\Lambda=0$, so assume $A_\Lambda>0$.

By Lemmas~\ref{lem:cap} and \ref{lem:density-cylinder}, choose a lateral point $Q_c$, a scale $r>0$, and a coordinate cylinder $Q$ such that
\[
 Q\subset\Phi(B_{r/4}\cap\{y_d=0\})\subset\Gamma_r(Q_c),
 \qquad
 S=J\cap Q,
 \qquad
 \abs S\ge m_J>0.
\]
Apply Lemma~\ref{lem:low-value} to $h=\partial_{\nu_g}U_\Lambda|_S$. There exists $S_0\subset S$ with
\begin{equation}
 \abs{S_0}\ge\frac{\abs S}{2}\ge\frac{m_J}{2},
 \label{eq:Dir-low-measure}
\end{equation}
and
\begin{equation}
 \esssup_{S_0}\abs{\partial_{\nu_g}U_\Lambda}
 \le C_JO_J.
 \label{eq:Dir-low-sup}
\end{equation}

Pull the capped domain back by the metric-normal chart and apply Lemma~\ref{lem:BZ} to the pullback metric and density. In the region where the cap agrees with the product manifold, the coordinate-invariance identity \eqref{eq:conormal-invariance} identifies the normalized boundary datum with $(\partial_{\nu_g}U_\Lambda)\circ\Phi$; equivalently, the coefficient flux differs from it only by a uniformly positive bounded factor. After the equivalent change of boundary measure, the low-value estimate \eqref{eq:Dir-low-sup} is therefore exactly the datum required by Lemma~\ref{lem:BZ}. Lemma~\ref{lem:regularity}, applied to a compact neighborhood of the closure of the local capped domain, verifies the $W^{1,\infty}$ hypothesis of Lemma~\ref{lem:BZ}; ordinary interior regularity handles the artificial cap. There exist a fixed interior half-cylinder $K$, a slightly larger fixed local set $K_1$, and $\alpha_J\in(0,1)$ such that
\begin{equation}
 \sup_K\bigl(\abs{U_\Lambda}+\abs{\nabla_GU_\Lambda}\bigr)
 \le C_JO_J^{\alpha_J}
 \left(\sup_{K_1}\abs{\nabla_GU_\Lambda}\right)^{1-\alpha_J}.
 \label{eq:Dir-propagation-1}
\end{equation}
By Lemma~\ref{lem:regularity},
\begin{equation}
 \sup_{K_1}\abs{\nabla_GU_\Lambda}
 \le e^{C\sqrt\Lambda}A_\Lambda.
 \label{eq:Dir-global-gradient}
\end{equation}
Consequently,
\begin{equation}
 \norm{U_\Lambda}_{L^\infty(K)}
 \le C_JO_J^{\alpha_J}
 \left(e^{C\sqrt\Lambda}A_\Lambda\right)^{1-\alpha_J}.
 \label{eq:Dir-value-control}
\end{equation}

Choose $\omega\times I_1\Subset K$. Lemma~\ref{lem:BM-product} gives
\begin{equation}
 A_\Lambda
 \le Ce^{C\sqrt\Lambda}\norm{U_\Lambda}_{L^2(\omega\times I_1)}
 \le C_Ke^{C\sqrt\Lambda}\norm{U_\Lambda}_{L^\infty(K)}.
 \label{eq:Dir-BM}
\end{equation}
Combining \eqref{eq:Dir-value-control} and \eqref{eq:Dir-BM},
\[
 A_\Lambda
 \le C_Je^{C_J\sqrt\Lambda}
 O_J^{\alpha_J}A_\Lambda^{1-\alpha_J}.
\]
Divide by $A_\Lambda^{1-\alpha_J}$ and take the power $1/\alpha_J$:
\begin{equation}
 A_\Lambda\le C_Je^{C_J\sqrt\Lambda}O_J.
 \label{eq:Dir-closed}
\end{equation}
Squaring proves \eqref{eq:main-dir-spec}.
\end{proof}

\begin{corollary}
\label{cor:uniform-product}
Fix a relatively open boundary patch $\Gamma\Subset\partial\M$ whose closure lies in finitely many $W^{2,\infty}$ coordinate charts and fix a nondegenerate interval $I_*\Subset I_s$. For every $\eta>0$, there exists $C_\eta>0$ such that, for every $\Lambda\ge1$, every finite family $\{a_j\}_{\lambda_j\le\Lambda}$, and every measurable $F\subset\Gamma$ satisfying $S_g(F)\ge\eta$,
\begin{equation}
 \sum_{\lambda_j\le\Lambda}\abs{a_j}^2
 \le C_\eta e^{C_\eta\sqrt\Lambda}
 \int_{F\times I_*}\abs{\partial_{\nu_g}U_\Lambda}^2\dd S_g\dd s.
 \label{eq:uniform-product}
\end{equation}
The constant is independent of the geometry of $F$.
\end{corollary}

\begin{proof}
Cover $\Gamma\times I_*$ by finitely many sufficiently small boundary cylinders $Q_1,\ldots,Q_N$, all at admissible scales for Lemma~\ref{lem:BZ}. Since
\[
 \abs{F\times I_*}\ge\eta\abs{I_*},
\]
at least one cylinder satisfies
\[
 \abs{(F\times I_*)\cap Q_\ell}
 \ge\frac{\eta\abs{I_*}}{N}=:m_\eta>0.
\]
After the low-value selection, at least $m_\eta/2$ measure remains. Repeat the proof of Theorem~\ref{thm:dir-spectral} in each candidate cylinder and take the maximum of the finitely many local constants. All remaining constants depend only on the fixed cover, $\eta$, and the geometry of $(\M,g)$.
\end{proof}

\section{Proof of Theorem~\ref{thm:neu-spectral}: the Neumann boundary spectral inequality}
\label{sec:neumann}

In this section we prove the Neumann boundary spectral inequality using the geometric and spectral estimates developed in Section~\ref{sec:dirichlet}. The additional difficulty is that the Neumann continuation estimate propagates the full boundary gradient, whereas the observation consists only of boundary values. Lemmas~\ref{lem:simplex}--\ref{lem:positive-measure-jet} recover the tangential first-order jet on a positive-measure subset, while the homogeneous Neumann condition supplies the normal component. A low-value point then determines the remaining additive constant, including the zero mode, and local spectral concentration completes the proof.

\begin{lemma}
\label{lem:simplex}
Let $G\subset B_\rho(x)\subset\R^q$ be measurable and satisfy
\[
 \abs G\ge\gamma\rho^q.
\]
Then there are $y_0,\ldots,y_q\in G$ such that the matrix
\[
 M=\begin{pmatrix}
 (y_1-y_0)^T\\
 \vdots\\
 (y_q-y_0)^T
 \end{pmatrix}
\]
is invertible and
\begin{equation}
 \abs{\det M}\ge c_{q,\gamma}\rho^q,
 \qquad
 \norm{M^{-1}}\le C_{q,\gamma}\rho^{-1}.
 \label{eq:simplex-bounds}
\end{equation}
\end{lemma}

\begin{proof}
Translate and scale to $x=0$ and $\rho=1$. Select the points inductively. Suppose $y_0,\ldots,y_k$ have been chosen and their affine span has dimension $k<q$. The $\delta$-neighborhood of any such affine plane inside $B_1$ has volume at most
\[
 C_q\delta^{q-k}.
\]
Choose $\delta_k=\delta_k(q,\gamma)>0$ so that this volume is less than $\abs G/2$, and select $y_{k+1}\in G$ at distance at least $\delta_k$ from the preceding affine span. The resulting simplex has volume bounded below by a positive constant depending only on $q$ and $\gamma$. The determinant estimate follows. The inverse estimate follows from the adjugate formula. Scaling restores the powers of $\rho$.
\end{proof}

\begin{lemma}
\label{lem:positive-measure-jet}
Let $Q\subset\R^q$ be a fixed cube, let $E\subset Q$ be measurable with $\abs E>0$, and let $0<\beta<1$. There exist $E^*\subset E$, $\abs{E^*}>0$, and $C_E>0$ such that every $f\in C^{1,\beta}(Q)$ satisfies
\begin{equation}
 \esssup_{E^*}\abs{\nabla f}
 \le C_E\norm f_{L^2(E)}^{\eta_0}
 \left(
  \norm{\nabla f}_{L^\infty(Q)}
  +[\nabla f]_{C^{0,\beta}(Q)}
 \right)^{1-\eta_0},
 \label{eq:positive-measure-jet}
\end{equation}
where
\begin{equation}
 \eta_0=\frac{\beta}{1+q/2+\beta}\in(0,1).
 \label{eq:eta-zero}
\end{equation}
\end{lemma}

\begin{proof}

Since $\partial Q$ has measure zero and
\[
 E\setminus\partial Q
 =\bigcup_{m\ge1}\{x\in E:\dist(x,\partial Q)>1/m\},
\]
there are a positive-measure set $E'\subset E$ and $d_E>0$ such that
\[
 \dist(E',\partial Q)\ge d_E.
\]
For $m\ge1$, define
\[
 E_m=\left\{x\in E':
 \abs{E\cap B_\rho(x)}\ge\frac{\omega_q}{2}\rho^q
 \text{ for every rational }0<\rho<1/m\right\}.
\]
The Lebesgue density theorem gives
\[
 \abs{E'\setminus\bigcup_{m\ge1}E_m}=0.
\]
Hence some $E_m$ has positive measure. By continuity of $\rho\mapsto\abs{E\cap B_\rho(x)}$, the same estimate holds for every real $0<\rho<1/m$. Choose such an $m$, put $E^*=E_m$, and take
\[
 r_E<\min\{1/m,d_E/2\}.
\]
Then there exists $\kappa>0$ such that
\begin{equation}
 \abs{E\cap B_\rho(x)}\ge\kappa\rho^q
 \quad(x\in E^*,\ 0<\rho<r_E).
 \label{eq:uniform-density-neumann}
\end{equation}
All balls used below lie in $Q$.

Fix $x\in E^*$ and $0<\rho<r_E$. Define
\[
 G_\rho=
 \left\{y\in E\cap B_\rho(x):
 \abs{f(y)}\le C\kappa^{-1/2}\rho^{-q/2}\norm f_{L^2(E)}
 \right\}.
\]
For a sufficiently large absolute $C$, Chebyshev's inequality and \eqref{eq:uniform-density-neumann} give
\[
 \abs{G_\rho}\ge c\kappa\rho^q.
\]
By Lemma~\ref{lem:simplex}, there are $y_0,\ldots,y_q\in G_\rho$ such that the matrix $M_\rho$ of differences satisfies
\begin{equation}
 \norm{M_\rho^{-1}}\le C_\kappa\rho^{-1}.
 \label{eq:simplex-inverse}
\end{equation}

For $j=1,\ldots,q$,
\[
 f(y_j)-f(y_0)
 =\nabla f(x)\cdot(y_j-y_0)+R_j,
 \qquad
 \abs{R_j}\le C[\nabla f]_{C^{0,\beta}(Q)}\rho^{1+\beta}.
\]
Solving the resulting linear system with \eqref{eq:simplex-inverse} yields
\begin{equation}
 \abs{\nabla f(x)}
 \le C_E\left(
  \rho^{-1-q/2}\norm f_{L^2(E)}
  +[\nabla f]_{C^{0,\beta}(Q)}\rho^\beta
 \right).
 \label{eq:jet-balance}
\end{equation}

Set
\[
 A=\norm f_{L^2(E)},
 \qquad
 B=\norm{\nabla f}_{L^\infty(Q)}+[\nabla f]_{C^{0,\beta}(Q)}.
\]
If $A=0$, then \eqref{eq:jet-balance} gives
\[
 \abs{\nabla f(x)}\le CB\rho^\beta
\]
for every $0<\rho<r_E$, and letting $\rho\downarrow0$ yields $\nabla f(x)=0$ on $E^*$. If $B=0$, the conclusion is immediate. Assume $A,B>0$. Balance the two terms in \eqref{eq:jet-balance} with
\[
 \rho\asymp(A/B)^{1/(1+q/2+\beta)}
\]
when this radius is at most $r_E$. This gives \eqref{eq:positive-measure-jet} with \eqref{eq:eta-zero}. If the balancing radius exceeds $r_E$, then $A/B$ has a fixed positive lower bound, and the trivial estimate $\abs{\nabla f}\le B$ gives the same conclusion after increasing the constant.
\end{proof}

\begin{proof}[Proof of Theorem~\ref{thm:neu-spectral}]
Set
\[
 B_\Lambda=\left(\sum_{\mu_j\le\Lambda}\abs{b_j}^2\right)^{1/2},
 \qquad
 O_J=\norm{W_\Lambda}_{L^2(J)}.
\]
The assertion is immediate when $B_\Lambda=0$, so assume $B_\Lambda>0$.

By Lemmas~\ref{lem:cap} and \ref{lem:density-cylinder}, choose a lateral point $Q_c$, a scale $r$, and a coordinate cylinder $Q$ such that
\[
 Q\subset\Phi(B_{r/4}\cap\{y_d=0\}),
 \qquad
 S=J\cap Q,
 \qquad
 \abs S>0.
\]
In flattened boundary coordinates let
\[
 f=W_\Lambda|_{\Sigma_s}.
\]
The lateral boundary has $q=n$ tangential variables: $n-1$ spatial tangential variables and the auxiliary variable $s$. The coordinate Lebesgue measure and $dS_g\,ds$ are mutually comparable. Apply Lemma~\ref{lem:positive-measure-jet} to $f$ and $E=S$. There exists a fixed positive-measure set $S^*\subset S$ such that
\begin{align}
 \esssup_{S^*}\abs{\nabla_\tau f}
 &\le C_JO_J^{\eta_0}
 \left(
  \norm{\nabla_\tau f}_{L^\infty(Q)}
  +[\nabla_\tau f]_{C^{0,\beta}(Q)}
 \right)^{1-\eta_0}
 \notag\\
 &\le C_JO_J^{\eta_0}
 \left(e^{C\sqrt\Lambda}B_\Lambda\right)^{1-\eta_0},
 \label{eq:Neumann-tangential-jet}
\end{align}
where the second inequality follows from Lemma~\ref{lem:regularity}.

By the homogeneous Neumann condition and \eqref{eq:Neu-reflection-gradient}, the full boundary gradient is equivalent to this tangential jet. Applying the Neumann part of Lemma~\ref{lem:BZ}, and then Lemma~\ref{lem:regularity}, gives
\begin{equation}
 \sup_{\Phi(B_{r/2}^+)}\abs{\nabla_GW_\Lambda}
 \le C_JO_J^\delta
 \left(e^{C\sqrt\Lambda}B_\Lambda\right)^{1-\delta},
 \qquad
 \delta=\alpha_J\eta_0\in(0,1).
 \label{eq:Neumann-gradient}
\end{equation}

It remains to fix the additive constant, including the Neumann zero mode. Apply Lemma~\ref{lem:low-value} to $W_\Lambda|_{S^*}$. Choose $z_0$ outside the resulting null exceptional set such that
\begin{equation}
 \abs{W_\Lambda(z_0)}
 \le\left(\frac2{\abs{S^*}}\right)^{1/2}
 \norm{W_\Lambda}_{L^2(S^*)}
 \le C_JO_J.
 \label{eq:Neumann-low-point}
\end{equation}
The continuous representative supplied by Lemma~\ref{lem:regularity} makes this point choice legitimate. Let
\[
 K\Subset\Phi(B_{r/4}^+).
\]
For every $y\in K$, Lemma~\ref{lem:cap} gives a curve from $z_0$ to $y$ of length at most $Cr$ contained in $\Phi(B_{r/2}^+)$. Integrating the gradient along the curve and using \eqref{eq:Neumann-gradient},
\begin{equation}
 \norm{W_\Lambda}_{L^\infty(K)}
 \le C_JO_J
 +C_JO_J^\delta
 \left(e^{C\sqrt\Lambda}B_\Lambda\right)^{1-\delta}.
 \label{eq:Neumann-value-pre}
\end{equation}
The trace theorem and Lemma~\ref{lem:energy} imply
\[
 O_J
 \le\norm{W_\Lambda}_{L^2(\Sigma_s)}
 \le Ce^{C\sqrt\Lambda}B_\Lambda.
\]
Hence the first term in \eqref{eq:Neumann-value-pre} is absorbed into the second and
\begin{equation}
 \norm{W_\Lambda}_{L^\infty(K)}
 \le C_JO_J^\delta
 \left(e^{C\sqrt\Lambda}B_\Lambda\right)^{1-\delta}.
 \label{eq:Neumann-value}
\end{equation}
Choose $\omega\times I_1\Subset K$. By Lemma~\ref{lem:BM-product},
\[
 B_\Lambda
 \le C_Je^{C_J\sqrt\Lambda}
 \norm{W_\Lambda}_{L^2(\omega\times I_1)}
 \le C_Je^{C_J\sqrt\Lambda}O_J^\delta B_\Lambda^{1-\delta}.
\]
Divide by $B_\Lambda^{1-\delta}$, square, and rename the constant to obtain \eqref{eq:main-neu-spec}.
\end{proof}

\begin{remark}
\label{rem:Neumann-mechanism}
The Neumann part of Lemma~\ref{lem:BZ} propagates the gradient, whereas the observation in Theorem~\ref{thm:neu-spectral} consists of function values. Lemma~\ref{lem:positive-measure-jet} supplies the missing tangential gradient information on a positive-measure subset. A single low-value point is then used to recover the additive constant, including the Neumann zero mode.
\end{remark}

\section{Proof of Theorem~\ref{thm:heat-observability}: measurable-boundary heat observability}
\label{sec:heat}

In this section we prove the measurable-boundary heat observability inequality from the Dirichlet spectral estimate. Since a fixed-time conormal-derivative spectral inequality is false in general, we recover the full time jet of the low-frequency boundary heat trace and relate it to the conormal trace of an auxiliary elliptic extension through Lemma~\ref{lem:time-jet-identity}. Proposition~\ref{prop:transfer} transfers the measurable heat observation to the elliptic trace. A low--high-frequency decomposition and frequency optimization yield the local interpolation estimate of Proposition~\ref{prop:one-step}. Finally, a density-point construction and a weighted telescoping argument prove Theorem~\ref{thm:heat-observability}.

From this point onward only the Dirichlet estimate of Corollary~\ref{cor:uniform-product} is used. Let
\[
 A=A_D=-\Delta_{g,D},
\]
and define
\[
 P_\Lambda f=\sum_{\lambda_j\le\Lambda}(f,\phi_j)_{L^2(\M)}\phi_j,
 \qquad
 Q_\Lambda=\Id-P_\Lambda.
\]
The heat solution is $u(t)=e^{-tA}u_0$. For an analytic function $h$, we refer to the sequence $(h^{(k)}(\tau))_{k\ge0}$ as its time jet at $\tau$.

In general one cannot assert a fixed-time estimate of the form
\[
 \norm v_{L^2(\M)}^2
 \le Ce^{C\sqrt\Lambda}
 \norm{\partial_{\nu_g}v}_{L^2(F)}^2,
 \qquad v\in\Ran P_\Lambda.
\]
Indeed, this already fails on the Euclidean unit disk. Dirichlet eigenfunctions may be written
\[
 \phi_{m,k}(r,\theta)=c_{m,k}J_m(j_{m,k}r)\cos(m\theta),
\]
where $j_{m,k}$ is the $k$th positive zero of $J_m$. Their boundary normal derivatives have the form
\[
 \partial_\nu\phi_{m,k}(1,\theta)=d_{m,k}\cos(m\theta),
 \qquad d_{m,k}\ne0.
\]
For two distinct radial indices $k_1\ne k_2$, the nonzero function
\[
 v=d_{m,k_2}\phi_{m,k_1}-d_{m,k_1}\phi_{m,k_2}
\]
has identically zero normal derivative on the entire boundary. The factors $\cosh(\sqrt\lambda s)$ in the auxiliary extension and $e^{-\lambda t}$ in the heat flow distinguish the two eigenvalues. Lemma~\ref{lem:time-jet-identity} is the bridge between these two modal-separation mechanisms.

\begin{lemma}
\label{lem:analytic-real}
Let $h$ be real analytic on $[a,a+L]$ and assume
\begin{equation}
 \abs{h^{(k)}(t)}
 \le M k!(2\rho L)^{-k},
 \qquad k\ge0,
 \quad t\in[a,a+L].
 \label{eq:analytic-deriv-bound}
\end{equation}
If $E\subset[a,a+L]$ is measurable with $\abs E>0$, there are
\[
 N=N(\rho,\abs E/L),
 \qquad
 \gamma=\gamma(\rho,\abs E/L)\in(0,1),
\]
such that
\begin{equation}
 \norm h_{L^\infty(a,a+L)}
 \le N\left(L^{-1}\int_E\abs{h(t)}\dd t\right)^\gamma M^{1-\gamma}.
 \label{eq:analytic-real-propagation}
\end{equation}
\end{lemma}

\begin{proof}
This is the translated and rescaled form of \cite[Lemma~13]{ApraizEtAl}. Their normalized average is equivalent to the one in \eqref{eq:analytic-real-propagation}, with constants depending on $\abs E/L$.
\end{proof}

\begin{lemma}
\label{lem:analytic-jet}
Let $h$ be analytic in the complex disk $D(\tau,4r)$ and satisfy $\abs{h(z)}\le M$. Let $E\subset(\tau-r,\tau+r)$ satisfy $\abs E\ge\kappa r$. Then there are $C_\kappa>0$, $c_\kappa\in(0,1)$, and $\theta_\kappa\in(0,1)$ such that, for every $k\ge0$,
\begin{equation}
 \abs{h^{(k)}(\tau)}
 \le C_\kappa
 \left(r^{-1}\int_E\abs{h(t)}^2\dd t\right)^{\theta_\kappa/2}
 M^{1-\theta_\kappa}
 \frac{k!}{(c_\kappa r)^k}.
 \label{eq:analytic-jet}
\end{equation}
\end{lemma}

\begin{proof}
For $t\in[\tau-r,\tau+r]$, the closed disk of radius $5r/2$ centered at $t$ is compactly contained in $D(\tau,4r)$. Cauchy's estimate gives
\[
 \abs{h^{(m)}(t)}\le M m!(5r/2)^{-m}.
\]
Apply Lemma~\ref{lem:analytic-real} on the interval of length $2r$; the analyticity parameter is a fixed positive number. Cauchy--Schwarz yields
\begin{equation}
 \sup_{\abs{t-\tau}<r}\abs{h(t)}
 \le C_\kappa
 \left(r^{-1}\int_E\abs{h(t)}^2\dd t\right)^{\theta_0/2}
 M^{1-\theta_0}=:m.
 \label{eq:analytic-bound-real}
\end{equation}
Set $m_0=\min\{m,M\}$. In the upper half-disk
\[
 D^+=D(\tau,r)\cap\{\operatorname{Im}z>0\},
\]
let $\omega(z)$ be the harmonic measure of the diameter. For a fixed $c_0\in(0,1/4)$, there exists $\theta_1>0$ such that
\[
 \omega(z)\ge\theta_1
 \quad\text{on }D(\tau,c_0r)\cap D^+.
\]
Apply the two-constants theorem to the subharmonic function
\[
 \frac12\log(\abs h^2+\varepsilon^2)
\]
and let $\varepsilon\downarrow0$. The same argument in the lower half-disk gives
\[
 \sup_{\abs{z-\tau}\le c_0r}\abs{h(z)}
 \le C_\kappa
 \left(r^{-1}\int_E\abs h^2\right)^{\theta_0\theta_1/2}
 M^{1-\theta_0\theta_1}.
\]
If the right-hand side in \eqref{eq:analytic-bound-real} is at least $M$, the same bound follows from $\abs h\le M$ after increasing the constant. Cauchy's formula on $\abs{z-\tau}=c_0r$ proves \eqref{eq:analytic-jet} with $\theta_\kappa=\theta_0\theta_1$.
\end{proof}

\begin{lemma}
\label{lem:time-window}
Let $\Gamma\Subset\partial\M$ be a fixed boundary patch, let $K\subset\Gamma\times(a,b)$ be measurable, and put $L=b-a$. Assume
\begin{equation}
 \abs K\ge\kappa S_g(\Gamma)L,
 \qquad 0<\kappa\le1.
 \label{eq:K-density}
\end{equation}
Then there are $\tau\in(a,b)$, $r\asymp_\kappa L$, a measurable set $F\subset\Gamma$, and the interior observation set
\begin{equation}
 K_f=K\cap\bigl(F\times(\tau-r,\tau+r)\bigr)
 \label{eq:Kf}
\end{equation}
such that
\begin{align}
 \tau-a&\ge8r,
 \label{eq:tau-away}\\
 S_g(F)&\ge c_\kappa S_g(\Gamma),
 \label{eq:F-measure}\\
 \abs{K_x}&\ge c_\kappa r
 \quad\text{for every }x\in F,
 \qquad
 K_x=\{t\in(\tau-r,\tau+r):(x,t)\in K\},
 \label{eq:fiber-measure}\\
 \abs{K_f}&\ge c_\kappa S_g(\Gamma)L.
 \label{eq:Kf-measure}
\end{align}
\end{lemma}

\begin{proof}
Remove the initial layer $\Gamma\times(a,a+\kappa L/4)$. Its measure is at most $\kappa S_g(\Gamma)L/4$, so the remaining part of $K$ has measure at least $3\kappa S_g(\Gamma)L/4$. Choose an integer $N=N(\kappa)$ satisfying
\[
 \frac{64}{\kappa}\le N<\frac{64}{\kappa}+1
\]
and divide $(a+\kappa L/4,b)$ into $N$ adjacent intervals of equal length $2r$. Then
\[
 2r=\frac{(1-\kappa/4)L}{N}\le\frac{\kappa L}{64},
 \qquad
 r\le\frac{\kappa L}{128},
\]
and also $r\asymp_\kappa L$. By the pigeonhole principle, one interval
\[
 I_0=(\tau-r,\tau+r)
\]
satisfies
\[
 \abs{K\cap(\Gamma\times I_0)}\ge c_\kappa S_g(\Gamma)r.
\]
Since $I_0$ lies after the removed initial layer,
\[
 \tau-a\ge\frac{\kappa L}{4}+r\ge33r>8r,
\]
which proves \eqref{eq:tau-away}.

The map $x\mapsto\abs{K_x}$ is measurable by Fubini. Choose a small fixed $\beta_\kappa>0$ and let
\[
 F=\{x\in\Gamma:\abs{K_x}\ge\beta_\kappa r\}.
\]
If $S_g(F)<c_\kappa S_g(\Gamma)$ with $c_\kappa$ sufficiently small, then
\[
 \abs{K\cap(\Gamma\times I_0)}
 \le2rS_g(F)+\beta_\kappa rS_g(\Gamma\setminus F),
\]
contradicting the preceding lower bound. Thus, after renaming constants, \eqref{eq:F-measure} and \eqref{eq:fiber-measure} hold. Integrating the fiber lower bound over $F$ gives \eqref{eq:Kf-measure}.
\end{proof}

Fix $a<b$, put $f=u(a)$ and $L=b-a$, and define
\begin{equation}
 g_\Lambda(x,t)
 =\partial_{\nu_g}P_\Lambda u(x,t)
 =\partial_{\nu_g}e^{-(t-a)A}P_\Lambda f.
 \label{eq:g-Lambda}
\end{equation}

\begin{lemma}
\label{lem:normal-trace}
There exists $C_\M>0$ such that, for every $v\in D(A)$,
\begin{equation}
 \norm{\partial_{\nu_g}v}_{L^2(\partial\M)}
 \le C_\M\norm{Av}_{L^2(\M)}.
 \label{eq:normal-trace}
\end{equation}
\end{lemma}

\begin{proof}
Let $F=Av\in L^2(\M)$. In each coordinate chart, $v$ satisfies a uniformly elliptic divergence-form equation with $W^{1,\infty}$ coefficients and an $L^2$ right-hand side. Interior difference quotients give $H^2$ regularity on smaller interior charts. On a boundary chart, use the $W^{2,\infty}$ flattening of Lemma~\ref{lem:cap}; the zero Dirichlet trace permits odd reflection of the solution and of the $L^2$ right-hand side, while the coefficient matrix is reflected as in \eqref{eq:reflection-matrix}. The reflected equation has $W^{1,\infty}$ coefficients and $L^2$ forcing, so the standard interior $H^2$ estimate applies. A finite partition of unity gives
\begin{equation}
 D(A)=H^2(\M)\cap H_0^1(\M),
 \qquad
 \norm v_{H^2(\M)}
 \le C\bigl(\norm{Av}_{L^2(\M)}+\norm v_{L^2(\M)}\bigr).
 \label{eq:global-H2}
\end{equation}
This is the global Dirichlet regularity theorem for a $W^{2,\infty}$ boundary and Lipschitz coefficients; see \cite[Chapters~8--9]{GilbargTrudinger} and \cite{LadyzhenskayaUraltseva}. The trace of $\nabla v\in H^1(\M)$ belongs to $H^{1/2}(\partial\M)$, and the metric normal is Lipschitz. Hence
\[
 \norm{\partial_{\nu_g}v}_{L^2(\partial\M)}
 \le C\norm v_{H^2(\M)}.
\]
The first Dirichlet eigenvalue is positive, so
\[
 \norm v_{L^2(\M)}\le\lambda_1^{-1}\norm{Av}_{L^2(\M)}.
\]
Combining these estimates proves \eqref{eq:normal-trace}.
\end{proof}

\begin{lemma}
\label{lem:majorant}
Assume $\tau-a\ge8r$ and define
\begin{equation}
 M_\Lambda(x)
 =\sum_{k=0}^\infty\frac{(4r)^k}{k!}
 \abs{\partial_t^kg_\Lambda(x,\tau)}.
 \label{eq:M-Lambda}
\end{equation}
Then
\begin{equation}
 \norm{M_\Lambda}_{L^2(\partial\M)}
 \le\frac{C}{\tau-a}\norm f_{L^2(\M)}
 \le\frac{C_\kappa}{L}\norm f_{L^2(\M)}.
 \label{eq:M-Lambda-bound}
\end{equation}
For almost every $x\in\partial\M$, the function $z\mapsto g_\Lambda(x,z)$ is entire and
\begin{equation}
 \abs{g_\Lambda(x,z)}\le M_\Lambda(x),
 \qquad \abs{z-\tau}<4r.
 \label{eq:M-majorizes}
\end{equation}
\end{lemma}

\begin{proof}
By Lemma~\ref{lem:normal-trace},
\begin{align*}
 \norm{\partial_t^kg_\Lambda(\cdot,\tau)}_{L^2(\partial\M)}
 &\le C\norm{A^{k+1}e^{-(\tau-a)A}P_\Lambda f}_{L^2(\M)}\\
 &\le C\sup_{\lambda\ge0}\lambda^{k+1}e^{-(\tau-a)\lambda}
 \norm f_{L^2(\M)}\\
 &\le C\frac{(k+1)!}{(\tau-a)^{k+1}}\norm f_{L^2(\M)}.
\end{align*}
Minkowski's inequality and $4r/(\tau-a)\le1/2$ give
\begin{align*}
 \norm{M_\Lambda}_{L^2(\partial\M)}
 &\le\frac{C}{\tau-a}
 \sum_{k=0}^\infty(k+1)
 \left(\frac{4r}{\tau-a}\right)^k
 \norm f_{L^2(\M)}\\
 &\le\frac{C}{\tau-a}\norm f_{L^2(\M)},
\end{align*}
which proves \eqref{eq:M-Lambda-bound}. The finite rank of $P_\Lambda$ makes $g_\Lambda$ an entire finite sum of exponentials. Choose common measurable representatives of the countably many derivatives in \eqref{eq:M-Lambda}; for almost every $x$, Taylor's formula gives \eqref{eq:M-majorizes}. If $M_\Lambda(x)=0$, the entire function is identically zero.
\end{proof}

Expand
\[
 f=u(a)=\sum_{j\ge1}f_j\phi_j.
\]
At the time $\tau$ supplied by Lemma~\ref{lem:time-window}, define
\begin{equation}
 V_{\Lambda,\tau}(x,s)
 =\sum_{\lambda_j\le\Lambda}
 f_je^{-(\tau-a)\lambda_j}\phi_j(x)
 \cosh(\sqrt{\lambda_j}s).
 \label{eq:V-Lambda-tau}
\end{equation}
Then
\begin{equation}
 \sum_{\lambda_j\le\Lambda}
 \abs{f_j}^2e^{-2(\tau-a)\lambda_j}
 =\norm{P_\Lambda u(\tau)}_{L^2(\M)}^2.
 \label{eq:V-coeff-mass}
\end{equation}

\begin{lemma}
\label{lem:time-jet-identity}
For every $s\in\R$, the identity
\begin{equation}
 \partial_{\nu_g}V_{\Lambda,\tau}(\cdot,s)
 =\sum_{k=0}^\infty
 \frac{(-1)^ks^{2k}}{(2k)!}
 \partial_t^kg_\Lambda(\cdot,\tau)
 \label{eq:time-jet-identity}
\end{equation}
holds in $L^2(\partial\M)$. After choosing common representatives, it holds for almost every boundary point for all $s$, and the series is absolutely convergent.
\end{lemma}

\begin{proof}
For each $k$,
\[
 \partial_t^kg_\Lambda(\cdot,\tau)
 =(-1)^k\sum_{\lambda_j\le\Lambda}
 \lambda_j^kf_je^{-(\tau-a)\lambda_j}
 \partial_{\nu_g}\phi_j,
\]
whereas
\[
 \cosh(\sqrt{\lambda_j}s)
 =\sum_{k\ge0}\frac{\lambda_j^ks^{2k}}{(2k)!}.
\]
The spectral sum is finite, so the two sums may be interchanged in $L^2(\partial\M)$; the two factors $(-1)^k$ cancel.
\end{proof}

\begin{proposition}
\label{prop:transfer}
Let $K\subset\Gamma\times(a,b)$ satisfy \eqref{eq:K-density}, and let $\tau,r,F,K_f$ be supplied by Lemma~\ref{lem:time-window}. Fix the auxiliary interval $I_*$ from Corollary~\ref{cor:uniform-product}. Then there are $C_\kappa>0$ and $\theta_\kappa\in(0,1)$ such that
\begin{equation}
 \norm{\partial_{\nu_g}V_{\Lambda,\tau}}_{L^2(F\times I_*)}
 \le C_\kappa e^{C_\kappa/L}
 \norm{g_\Lambda}_{L^2(K_f)}^{\theta_\kappa}
 \norm{u(a)}_{L^2(\M)}^{1-\theta_\kappa}.
 \label{eq:transfer}
\end{equation}
\end{proposition}

\begin{proof}
For almost every $x\in F$, the time fiber $K_x$ satisfies $\abs{K_x}\ge c_\kappa r$. Apply Lemma~\ref{lem:analytic-jet} to
\[
 h_x(z)=g_\Lambda(x,z)
\]
with majorant $M_\Lambda(x)$. For every $k\ge0$,
\begin{equation}
 \abs{\partial_t^kg_\Lambda(x,\tau)}
 \le C_\kappa A_\Lambda(x)^{\theta_\kappa}
 M_\Lambda(x)^{1-\theta_\kappa}
 \frac{k!}{(c_\kappa r)^k},
 \label{eq:pointwise-jet}
\end{equation}
where
\begin{equation}
 A_\Lambda(x)
 =\left(r^{-1}\int_{K_x}\abs{g_\Lambda(x,t)}^2\dd t\right)^{1/2}.
 \label{eq:A-Lambda-x}
\end{equation}
Substitution into \eqref{eq:time-jet-identity} gives
\begin{align*}
 \abs{\partial_{\nu_g}V_{\Lambda,\tau}(x,s)}
 &\le C_\kappa A_\Lambda(x)^{\theta_\kappa}
 M_\Lambda(x)^{1-\theta_\kappa}
 \sum_{k=0}^\infty\frac{k!}{(2k)!}
 \left(\frac{s^2}{c_\kappa r}\right)^k.
\end{align*}
Since $(2k)!\ge(k!)^2$, the series is bounded by
\[
 \exp\left(\frac{s^2}{c_\kappa r}\right)
 \le e^{C_\kappa/L}
\]
on the fixed interval $I_*$. H\"older's inequality in $x$ yields
\begin{equation}
 \norm{\partial_{\nu_g}V_{\Lambda,\tau}}_{L^2(F\times I_*)}
 \le C_\kappa e^{C_\kappa/L}
 \norm{A_\Lambda}_{L^2(F)}^{\theta_\kappa}
 \norm{M_\Lambda}_{L^2(F)}^{1-\theta_\kappa}.
 \label{eq:transfer-holder}
\end{equation}
By the definition of $K_f$ and Fubini,
\begin{equation}
 \norm{A_\Lambda}_{L^2(F)}^2
 =r^{-1}\int_{K_f}\abs{g_\Lambda}^2\dd S_g\dd t.
 \label{eq:A-Lambda-L2}
\end{equation}
Use Lemma~\ref{lem:majorant}, $r\asymp_\kappa L$, and absorb the resulting fixed powers of $L^{-1}$ into $e^{C_\kappa/L}$. This gives \eqref{eq:transfer}.
\end{proof}

Apply Corollary~\ref{cor:uniform-product} to $V_{\Lambda,\tau}$. By \eqref{eq:F-measure}, the spectral constant depends only on $\kappa$ and the fixed geometry. Combining \eqref{eq:V-coeff-mass} and Proposition~\ref{prop:transfer},
\begin{align}
 \norm{P_\Lambda u(\tau)}_{L^2(\M)}
 &\le C_\kappa e^{C_\kappa\sqrt\Lambda}
 \norm{\partial_{\nu_g}V_{\Lambda,\tau}}_{L^2(F\times I_*)}
 \notag\\
 &\le C_\kappa e^{C_\kappa/L+C_\kappa\sqrt\Lambda}
 \norm{g_\Lambda}_{L^2(K_f)}^\theta
 \norm{u(a)}_{L^2(\M)}^{1-\theta},
 \label{eq:low-frequency-at-tau}
\end{align}
where $\theta=\theta_\kappa$. Since $P_\Lambda$ commutes with the semigroup,
\begin{equation}
 \norm{P_\Lambda u(b)}_{L^2(\M)}
 \le\norm{P_\Lambda u(\tau)}_{L^2(\M)}.
 \label{eq:low-frequency-contract}
\end{equation}

Let $g=\partial_{\nu_g}u$. Then
\[
 g_\Lambda=g-\partial_{\nu_g}Q_\Lambda u.
\]
Every $(x,t)\in K_f$ has $t\in(\tau-r,\tau+r)$, and \eqref{eq:tau-away} implies $t-a\ge c_\kappa L$. By Lemma~\ref{lem:normal-trace},
\begin{align}
 \norm{\partial_{\nu_g}Q_\Lambda u(t)}_{L^2(\partial\M)}
 &\le C\norm{Ae^{-(t-a)A}Q_\Lambda u(a)}_{L^2(\M)}\notag\\
 &\le C\sup_{\lambda>\Lambda}\lambda e^{-(t-a)\lambda}
 \norm{u(a)}_{L^2(\M)}.
 \label{eq:high-trace-pre}
\end{align}
Writing $d=t-a$,
\[
 \lambda e^{-d\lambda}
 =e^{-d\lambda/2}\lambda e^{-d\lambda/2}
 \le\frac{C}{d}e^{-d\Lambda/2}.
\]
Thus on the interior time window,
\begin{equation}
 \norm{\partial_{\nu_g}Q_\Lambda u(t)}_{L^2(\partial\M)}
 \le C_\kappa L^{-1}e^{-c_\kappa L\Lambda}
 \norm{u(a)}_{L^2(\M)}.
 \label{eq:high-trace}
\end{equation}
Integrating only over $K_f\subset F\times(\tau-r,\tau+r)$ gives
\begin{equation}
 \norm{g_\Lambda}_{L^2(K_f)}
 \le\norm g_{L^2(K_f)}
 +C_\kappa L^{-1/2}e^{-c_\kappa L\Lambda}
 \norm{u(a)}_{L^2(\M)}.
 \label{eq:g-Lambda-high}
\end{equation}
The high-frequency component of the final state satisfies
\begin{equation}
 \norm{Q_\Lambda u(b)}_{L^2(\M)}
 \le e^{-L\Lambda}\norm{u(a)}_{L^2(\M)}.
 \label{eq:high-final}
\end{equation}
Substitute \eqref{eq:g-Lambda-high} into \eqref{eq:low-frequency-at-tau}, use $(X+Y)^\theta\le X^\theta+Y^\theta$, and combine with \eqref{eq:low-frequency-contract}--\eqref{eq:high-final}. Fixed powers of $L^{-1}$ are absorbed into $e^{C_\kappa/L}$. Young's inequality
\[
 C_\kappa\sqrt\Lambda
 \le\frac{c_\kappa\theta}{2}L\Lambda+\frac{C_\kappa}{L}
\]
yields the following estimate.

\begin{proposition}
\label{prop:bridge}
Under the assumptions of Lemma~\ref{lem:time-window}, for every $\Lambda\ge1$,
\begin{align}
 \norm{u(b)}_{L^2(\M)}
 &\le C_\kappa e^{C_\kappa/L+C_\kappa\sqrt\Lambda}
 \norm{\partial_{\nu_g}u}_{L^2(K_f)}^\theta
 \norm{u(a)}_{L^2(\M)}^{1-\theta}
 \notag\\
 &\quad+C_\kappa e^{C_\kappa/L-c_\kappa L\Lambda}
 \norm{u(a)}_{L^2(\M)}.
 \label{eq:bridge}
\end{align}
\end{proposition}

\begin{lemma}
\label{lem:frequency-optimization}
Let $0<L\le1$ and $0\le Y\le M$. Assume that, for every $\Lambda\ge1$,
\begin{equation}
 Y
 \le Ce^{C/L+C\sqrt\Lambda}O^\theta M^{1-\theta}
 +Ce^{C/L-cL\Lambda}M,
 \qquad 0<\theta<1.
 \label{eq:frequency-optimization-assumption}
\end{equation}
Then there are $\gamma\in(0,\theta)$ and $C>0$ such that
\begin{equation}
 Y\le Ce^{C/L}O^\gamma M^{1-\gamma}.
 \label{eq:frequency-optimized}
\end{equation}
\end{lemma}

\begin{proof}
If $M=0$, then $Y=0$. If $O=0$, let $\Lambda\to\infty$ in \eqref{eq:frequency-optimization-assumption}. Assume $M,O>0$ and set $q=O/M$. If $q\ge1$, then $Y\le M\le q^\gamma M$.

Let $0<q<1$ and put $R=\log(1/q)$. Fix $\gamma\in(0,\theta)$ and choose
\[
 \Lambda=\max\left\{1,\frac{\alpha R}{L}\right\},
\]
where $\alpha$ will be fixed. If the second term is active, then
\[
 C\sqrt\Lambda
 \le\frac{\theta-\gamma}{2}R+\frac{C_{\alpha,\theta,\gamma}}{L}.
\]
After division by $M$, the first term in \eqref{eq:frequency-optimization-assumption} is bounded by $Ce^{C/L}q^\gamma$. Taking $\alpha$ sufficiently large that $c\alpha\ge\gamma$ gives the same bound for the second term. If $\Lambda=1$, then $R\le L/\alpha\le1/\alpha$, so $q^\gamma$ is bounded below by a fixed positive number and $Y\le M$ proves the result.
\end{proof}

\begin{proposition}
\label{prop:one-step}
For every $\kappa>0$, there are $C_\kappa>0$ and $\gamma_\kappa\in(0,1)$ such that, whenever $0<b-a\le1$ and $K\subset\Gamma\times(a,b)$ satisfies
\[
 \abs K\ge\kappa S_g(\Gamma)(b-a),
\]
one has
\begin{equation}
 \norm{u(b)}_{L^2(\M)}
 \le C_\kappa\exp\left(\frac{C_\kappa}{b-a}\right)
 \norm{\partial_{\nu_g}u}_{L^2(K)}^{\gamma_\kappa}
 \norm{u(a)}_{L^2(\M)}^{1-\gamma_\kappa}.
 \label{eq:one-step}
\end{equation}
\end{proposition}

\begin{proof}
Apply Lemma~\ref{lem:frequency-optimization} to Proposition~\ref{prop:bridge} with
\[
 Y=\norm{u(b)}_{L^2(\M)},
 \quad
 M=\norm{u(a)}_{L^2(\M)},
 \quad
 O=\norm{\partial_{\nu_g}u}_{L^2(K_f)},
 \quad
 L=b-a.
\]
Contractivity gives $Y\le M$. Finally, $K_f\subset K$ implies
\[
 \norm{\partial_{\nu_g}u}_{L^2(K_f)}
 \le\norm{\partial_{\nu_g}u}_{L^2(K)}.
\]
\end{proof}

\begin{proof}[Proof of Theorem~\ref{thm:heat-observability}]
Cover $\partial\M$ by finitely many sufficiently small boundary patches. For at least one patch $\Gamma\Subset\partial\M$,
\[
 J^\Gamma=J\cap(\Gamma\times(0,T))
\]
has positive measure. For almost every $t$, define
\[
 J_t=\{x\in\Gamma:(x,t)\in J^\Gamma\},
 \qquad q(t)=S_g(J_t),
\]
and set
\begin{equation}
 m_0=\frac{\abs{J^\Gamma}}{2T}>0,
 \qquad
 E=\{t\in(0,T):q(t)\ge m_0\}.
 \label{eq:E-times}
\end{equation}
Fubini gives
\[
 \abs{J^\Gamma}
 \le S_g(\Gamma)\abs E+m_0T
 =S_g(\Gamma)\abs E+\frac{\abs{J^\Gamma}}2,
\]
so
\begin{equation}
 \abs E\ge\frac{\abs{J^\Gamma}}{2S_g(\Gamma)}>0.
 \label{eq:E-positive}
\end{equation}
Set
\[
 \kappa_0=\frac{m_0}{2S_g(\Gamma)}.
\]
Let $C>0$ and $\gamma\in(0,1)$ be the constants in Proposition~\ref{prop:one-step} for $\kappa_0$, and put
\[
 \alpha=\frac{1-\gamma}{\gamma}.
\]
Choose $z>1$ sufficiently close to one that
\begin{equation}
 \alpha(z-1)<\frac14.
 \label{eq:z-choice}
\end{equation}
Choose a right Lebesgue density point $\ell\in(0,T)$ of $E$. With
\[
 \varepsilon=\frac{1-z^{-1}}4,
\]
choose $\ell_1\in(\ell,T)$ so close to $\ell$ that $\ell_1-\ell<1$ and
\[
 \abs{(\ell,\ell+h)\setminus E}\le\varepsilon h
 \quad\text{for every }0<h\le\ell_1-\ell.
\]
Define
\begin{equation}
 \ell_m=\ell+z^{-(m-1)}(\ell_1-\ell),
 \qquad m\ge1,
 \label{eq:ell-m}
\end{equation}
and
\begin{equation}
 L_m=\ell_m-\ell_{m+1}
 =(1-z^{-1})z^{-(m-1)}(\ell_1-\ell),
 \qquad
 L_{m+1}=z^{-1}L_m.
 \label{eq:L-m}
\end{equation}
Since
\[
 \ell_m-\ell=\frac{L_m}{1-z^{-1}},
\]
the density estimate gives
\begin{equation}
 \abs{E\cap(\ell_{m+1},\ell_m)}
 \ge L_m-\varepsilon(\ell_m-\ell)
 \ge\frac12L_m.
 \label{eq:E-block}
\end{equation}
Define the pairwise disjoint observation blocks
\begin{equation}
 K_m=J^\Gamma\cap\bigl(\Gamma\times(\ell_{m+1},\ell_m)\bigr).
 \label{eq:K-m}
\end{equation}
Then
\begin{equation}
 \abs{K_m}
 =\int_{\ell_{m+1}}^{\ell_m}q(t)\dd t
 \ge m_0\abs{E\cap(\ell_{m+1},\ell_m)}
 \ge\kappa_0S_g(\Gamma)L_m.
 \label{eq:K-m-density}
\end{equation}
Thus Proposition~\ref{prop:one-step} applies on every interval with constants independent of $m$. Set
\[
 Y_m=\norm{u(\ell_m)}_{L^2(\M)},
 \qquad
 O_m=\norm{\partial_{\nu_g}u}_{L^2(K_m)}.
\]
Then
\begin{equation}
 Y_m\le Ce^{C/L_m}O_m^\gamma Y_{m+1}^{1-\gamma}.
 \label{eq:Y-m}
\end{equation}
Weighted Young's inequality gives, for every $\varepsilon_m>0$,
\begin{equation}
 Y_m
 \le\varepsilon_mY_{m+1}
 +Ce^{C_0/L_m}\varepsilon_m^{-\alpha}O_m,
 \qquad
 \alpha=\frac{1-\gamma}{\gamma}.
 \label{eq:weighted-Young}
\end{equation}
Choose $K_0>0$ large and define
\begin{equation}
 w_m=e^{-K_0/L_m},
 \qquad
 \varepsilon_m=e^{-K_0(z-1)/L_m}.
 \label{eq:weights}
\end{equation}
By \eqref{eq:L-m},
\begin{equation}
 w_m\varepsilon_m
 =e^{-K_0z/L_m}
 =e^{-K_0/L_{m+1}}
 =w_{m+1}.
 \label{eq:weight-identity}
\end{equation}
Multiplying \eqref{eq:weighted-Young} by $w_m$ gives
\begin{equation}
 w_mY_m-w_{m+1}Y_{m+1}
 \le Ce^{-D/L_m}O_m,
 \qquad
 D=K_0-C_0-\alpha K_0(z-1)>0,
 \label{eq:telescope-step}
\end{equation}
where positivity follows from \eqref{eq:z-choice} for $K_0$ sufficiently large. Summing from $m=1$ to $N$,
\begin{equation}
 w_1Y_1-w_{N+1}Y_{N+1}
 \le C\sum_{m=1}^Ne^{-D/L_m}O_m.
 \label{eq:telescope-finite}
\end{equation}
Since $L_m\downarrow0$, $w_m\downarrow0$, and contractivity gives $Y_m\le\norm{u_0}_{L^2(\M)}$, the terminal term tends to zero. The blocks $K_m$ are disjoint, so Cauchy--Schwarz gives
\begin{align}
 \sum_{m=1}^\infty e^{-D/L_m}O_m
 &\le\left(\sum_{m=1}^\infty e^{-2D/L_m}\right)^{1/2}
 \left(\sum_{m=1}^\infty O_m^2\right)^{1/2}
 \notag\\
 &\le C_{D,L_1,z}\norm{\partial_{\nu_g}u}_{L^2(J)}.
 \label{eq:weighted-sum}
\end{align}
The first series converges superexponentially because $L_m=L_1z^{-(m-1)}$. Letting $N\to\infty$ in \eqref{eq:telescope-finite} and dividing by the fixed number $w_1=e^{-K_0/L_1}>0$ yields
\begin{equation}
 \norm{u(\ell_1)}_{L^2(\M)}
 \le C_{J,T}\norm{\partial_{\nu_g}u}_{L^2(J)}.
 \label{eq:ell-one-observation}
\end{equation}
Finally, $\ell_1<T$ and the heat semigroup is contractive, so
\[
 \norm{u(T)}_{L^2(\M)}\le\norm{u(\ell_1)}_{L^2(\M)}.
\]
Squaring \eqref{eq:ell-one-observation} proves Theorem~\ref{thm:heat-observability}.
\end{proof}

\section{Proof of Theorem~\ref{thm:localized}: the localized manifold theorem}
\label{sec:localized}

In this section we localize the Dirichlet argument on a compact connected Lipschitz Riemannian manifold whose boundary is $W^{2,\infty}$ only near the observed patch. Thus, in this section only, the global $W^{2,\infty}$ boundary regularity imposed in the preceding manifold results is relaxed. The global condition $(\mathrm{LR})_D$ supplies low-frequency spectral concentration, while the local $W^{2,\infty}$ hypothesis provides the boundary continuation, regularity, and conormal-trace estimates. These ingredients yield the localized boundary spectral inequality and the local interpolation estimate of Proposition~\ref{prop:local-one-step}. The time-analyticity, frequency-optimization, and telescoping arguments of Section~\ref{sec:heat} then apply without change and prove Theorem~\ref{thm:localized}.

Let $(\M,g)$ be a compact connected Lipschitz Riemannian manifold with nonempty boundary, equipped in a finite Lipschitz atlas with a symmetric uniformly elliptic Lipschitz metric. Denote by $A_D$ the nonnegative self-adjoint Dirichlet Laplacian associated with the closed form
\[
 \mathfrak a(v,w)=\int_\M \ip{\nabla_gv}{\nabla_gw}_g\dd V_g,
 \qquad
 D(\mathfrak a)=H_0^1(\M),
\]
and let $(\lambda_j,\varphi_j)_{j\ge1}$ be its $L^2(\M,dV_g)$-orthonormal eigenpairs. For $\Lambda>0$, write
\[
 P_\Lambda^Df=\sum_{\lambda_j\le\Lambda}(f,\varphi_j)_{L^2(\M)}\varphi_j.
\]
We shall use the following intrinsic Dirichlet Lebeau--Robbiano property, denoted by $(\mathrm{LR})_D$:
\begin{equation}
 \norm{P_\Lambda^Df}_{L^2(\M)}
 \le C_\omega e^{C_\omega\sqrt\Lambda}
 \norm{P_\Lambda^Df}_{L^2(\omega)}
 \label{eq:local-LR}
\end{equation}
for every fixed nonempty open set $\omega\Subset\M^\circ$, every $\Lambda\ge1$, and every $f\in L^2(\M)$. The constant $C_\omega$ may depend on $(\M,g)$ and $\omega$, but is uniform in $\Lambda$ and $f$. This is precisely the form of global spectral concentration needed in the proof below. In the Euclidean case $\M=\Omega\subset\R^n$, the interior-ball condition \cite[(1.4)]{ApraizEtAl} implies \eqref{eq:local-LR}: given $\omega\Subset\Omega$, choose $B_{4R}(x_*)\Subset\omega$ and apply that condition on $B_R(x_*)$. In particular, \cite[Theorem~3]{ApraizEtAl} yields \eqref{eq:local-LR} for bounded Lipschitz domains that are locally star-shaped.

\begin{lemma}
\label{lem:local-BM}
Assume \eqref{eq:local-LR} and let $\omega\Subset\M^\circ$ be a fixed nonempty open set. There exists $C_\omega>0$ such that, for every $\Lambda\ge1$ and every finite Dirichlet sum
\[
 f_\Lambda=\sum_{\lambda_j\le\Lambda}c_j\varphi_j,
\]
one has
\begin{equation}
 \sum_{\lambda_j\le\Lambda}\abs{c_j}^2
 \le C_\omega e^{C_\omega\sqrt\Lambda}
 \norm{f_\Lambda}_{L^2(\omega)}^2.
 \label{eq:local-open-spectral}
\end{equation}
Moreover, if $I_1\Subset I_s$ is a nondegenerate interval and
\[
 V_\Lambda(x,s)=\sum_{\lambda_j\le\Lambda}c_j\varphi_j(x)\cosh(\sqrt{\lambda_j}s),
\]
then
\begin{equation}
 \sum_{\lambda_j\le\Lambda}\abs{c_j}^2
 \le C_{\omega,I_1}e^{C_{\omega,I_1}\sqrt\Lambda}
 \norm{V_\Lambda}_{L^2(\omega\times I_1)}^2.
 \label{eq:local-product-spectral}
\end{equation}
\end{lemma}

\begin{proof}
Since $f_\Lambda=P_\Lambda^Df_\Lambda$, \eqref{eq:local-LR} gives
\[
 \norm{f_\Lambda}_{L^2(\M)}
 \le C_\omega e^{C_\omega\sqrt\Lambda}
 \norm{f_\Lambda}_{L^2(\omega)}.
\]
By orthonormality,
\[
 \norm{f_\Lambda}_{L^2(\M)}^2
 =\sum_{\lambda_j\le\Lambda}\abs{c_j}^2.
\]
Squaring and renaming the constant proves \eqref{eq:local-open-spectral}. For each fixed $s\in I_1$, apply \eqref{eq:local-open-spectral} to the low-frequency sum $V_\Lambda(\cdot,s)$, integrate in $s$, and use $\cosh^2(\sqrt{\lambda_j}s)\ge1$. Absorb $\abs{I_1}^{-1}$ into the constant.
\end{proof}

Let $\Gamma_0,\Gamma_1\subset\partial\M$ be relatively open and satisfy
\[
 \Gamma_0\Subset\Gamma_1
\]
in the relative topology of $\partial\M$. Assume there exists a neighborhood $U_1$ of $\Gamma_1$ in $\overline\M$ covered by finitely many boundary charts in which the boundary is represented by $W^{2,\infty}(C^{1,1})$ graphs, with uniform local constants. Since the metric is Lipschitz and uniformly elliptic, its coefficients in these charts have the same regularity class required by the local arguments of Section~\ref{sec:dirichlet}.

\begin{lemma}
\label{lem:local-geometry}
Under the preceding local $W^{2,\infty}$ assumption, the following conclusions hold with constants uniform for points of $\Gamma_0$.
\begin{enumerate}[label=\textnormal{(\roman*)}]
\item The capped-domain construction of Lemma~\ref{lem:cap} can be carried out at every point of $\Gamma_0\times I_*$, for every fixed $I_*\Subset I_s$.
\item The Dirichlet odd reflection of Lemma~\ref{lem:odd-reflection} has uniformly elliptic $W^{1,\infty}$ coefficients, and the measurable-boundary continuation estimate of Lemma~\ref{lem:BZ} applies on the resulting local capped domains.
\item If $K$ is a compact subset of $(U_0\cap\M)\times I_s$, where $U_0\Subset U_1$, and $K$ is separated from $s=\pm s_0$, then the Dirichlet extension satisfies
\begin{equation}
 \norm{V_\Lambda}_{C^{1,\sigma}(K)}
 \le C_Ke^{C_K\sqrt\Lambda}
 \left(\sum_{\lambda_j\le\Lambda}\abs{c_j}^2\right)^{1/2}
 \label{eq:local-C1sigma}
\end{equation}
for every fixed $0<\sigma<1$.
\end{enumerate}
\end{lemma}

\begin{proof}
All three assertions are local. Since $\Gamma_0\Subset\Gamma_1$, compactness gives a finite family of $W^{2,\infty}$ boundary charts whose slightly enlarged coordinate neighborhoods remain in $U_1$. The metric-normal capping construction of Lemma~\ref{lem:cap} uses only one such chart and a scale smaller than its radius. Hence the cap agrees with $\M\times I_s$ near the chosen lateral point and, after pullback to the chart, is a bounded $W^{2,\infty}$ domain with uniform rescaled geometry.

In these metric-normal coordinates, the flattening map belongs to $W^{2,\infty}$, so the coefficient matrix of the transformed product operator belongs to $W^{1,\infty}$. As in \eqref{eq:reflection-matrix}, the tangential--normal blocks vanish on the flat boundary, and the blockwise even/odd extension remains Lipschitz. Uniform ellipticity is preserved. This is exactly the coefficient class required by the boundary continuation theorem.

For \eqref{eq:local-C1sigma}, cover $K$ by finitely many interior coordinate balls and boundary half-balls contained in $U_1\times I_s$. On a boundary half-ball, flatten and reflect as in Lemma~\ref{lem:odd-reflection}. The reflected function solves a uniformly elliptic divergence-form equation with Lipschitz coefficients. Fix $0<\sigma<1$. The reflected coefficient matrix belongs to $C^{0,\sigma}$ with a uniform norm, and the direct interior $C^{1,\sigma}$ estimate used in Lemma~\ref{lem:regularity} gives the local bound from the $L^2$ norm. The only global input is the energy of the finite spectral sum, which follows from orthogonality and the Dirichlet form identity and does not require global $W^{2,\infty}$ boundary regularity:
\[
 \norm{V_\Lambda}_{L^2(\M\times I_s)}
 \le Ce^{s_0\sqrt\Lambda}
 \left(\sum_{\lambda_j\le\Lambda}\abs{c_j}^2\right)^{1/2}.
\]
Taking the maximum of the finitely many chart constants proves the result.
\end{proof}

\begin{proposition}
\label{prop:local-spectral}
Assume that $(\M,g)$ is compact, connected, and globally Lipschitz, satisfies \eqref{eq:local-LR}, and that $\Gamma_0\Subset\Gamma_1$ are as above, with boundary of class $W^{2,\infty}$ near $\Gamma_1$. If $J\subset\Gamma_0\times I_s$ is measurable with $\abs J>0$, then there exists $C_J>0$ such that, for every $\Lambda\ge1$ and every finite family $\{a_j\}_{\lambda_j\le\Lambda}$,
\begin{equation}
 \sum_{\lambda_j\le\Lambda}\abs{a_j}^2
 \le C_Je^{C_J\sqrt\Lambda}
 \int_J\abs{\partial_{\nu_g}U_\Lambda}^2\dd S_g\dd s.
 \label{eq:local-boundary-spectral}
\end{equation}
\end{proposition}

\begin{proof}
Set
\[
 A_\Lambda=\left(\sum_{\lambda_j\le\Lambda}\abs{a_j}^2\right)^{1/2},
 \qquad
 O_J=\norm{\partial_{\nu_g}U_\Lambda}_{L^2(J)}.
\]
The claim is immediate if $A_\Lambda=0$, so assume $A_\Lambda>0$.

\emph{Step 1.}
Choose a density point $(x_0,s_*)$ of $J$. Since $J\subset\Gamma_0\times I_s$ and $\Gamma_0$ is relatively open, the boundary space--time cylinder can be chosen so small that its closure lies in $\Gamma_0\times I_s$ and is separated from $s=\pm s_0$. Thus there exist a boundary coordinate cylinder $Q$, a local capped domain, and a set $S=J\cap Q$ with $\abs S>0$. Applying Lemma~\ref{lem:low-value} to $\partial_{\nu_g}U_\Lambda|_S$, we obtain $S_0\subset S$ such that
\[
 \abs{S_0}\ge\frac{\abs S}{2}>0,
 \qquad
 \esssup_{S_0}\abs{\partial_{\nu_g}U_\Lambda}
 \le C_JO_J.
\]
The lower bound for $\abs{S_0}$ depends only on the fixed set $J$ and the chosen cylinder, not on $\Lambda$ or the coefficients.

\emph{Step 2.}
By Lemma~\ref{lem:local-geometry}, the capped domain is $W^{2,\infty}$ in local coordinates and agrees with $\M\times I_s$ near $S_0$. Applying the Dirichlet part of Lemma~\ref{lem:BZ} to the pulled-back product metric and density, we obtain a fixed interior half-cylinder $K$, a slightly larger local set $K_1$, and $\alpha_J\in(0,1)$ such that
\[
 \sup_K\bigl(\abs{U_\Lambda}+\abs{\nabla_GU_\Lambda}\bigr)
 \le C_JO_J^{\alpha_J}
 \left(\sup_{K_1}\abs{\nabla_GU_\Lambda}\right)^{1-\alpha_J},
 \qquad G=g\oplus ds^2.
\]
Here, as in \eqref{eq:conormal-invariance}, the normalized conormal derivative is invariant under the boundary coordinate change, and the coordinate surface measure is uniformly equivalent to $dS_g\,ds$. The localized regularity estimate \eqref{eq:local-C1sigma} gives
\[
 \sup_{K_1}\abs{\nabla_GU_\Lambda}
 \le Ce^{C\sqrt\Lambda}A_\Lambda.
\]
Consequently,
\begin{equation}
 \norm{U_\Lambda}_{L^\infty(K)}
 \le C_JO_J^{\alpha_J}
 \left(e^{C\sqrt\Lambda}A_\Lambda\right)^{1-\alpha_J}.
 \label{eq:local-value-control}
\end{equation}

\emph{Step 3.}
Choose $\omega\times I_1\Subset K$, with $\omega\Subset\M^\circ$ nonempty. Lemma~\ref{lem:local-BM} gives
\[
 A_\Lambda
 \le Ce^{C\sqrt\Lambda}
 \norm{U_\Lambda}_{L^2(\omega\times I_1)}
 \le C_Ke^{C\sqrt\Lambda}
 \norm{U_\Lambda}_{L^\infty(K)}.
\]
Combining this with \eqref{eq:local-value-control},
\[
 A_\Lambda
 \le C_Je^{C_J\sqrt\Lambda}O_J^{\alpha_J}A_\Lambda^{1-\alpha_J}.
\]
Division by $A_\Lambda^{1-\alpha_J}$ and raising to the power $1/\alpha_J$ give
\[
 A_\Lambda\le C_Je^{C_J\sqrt\Lambda}O_J.
\]
Squaring proves \eqref{eq:local-boundary-spectral}.
\end{proof}

\begin{corollary}
\label{cor:local-product}
Under the hypotheses of Proposition~\ref{prop:local-spectral}, fix a nondegenerate interval $I_*\Subset I_s$. For every $\eta>0$, there exists $C_\eta>0$ such that, for every $\Lambda\ge1$, every finite family $\{a_j\}_{\lambda_j\le\Lambda}$, and every measurable $F\subset\Gamma_0$ satisfying $S_g(F)\ge\eta$,
\begin{equation}
 \sum_{\lambda_j\le\Lambda}\abs{a_j}^2
 \le C_\eta e^{C_\eta\sqrt\Lambda}
 \int_{F\times I_*}\abs{\partial_{\nu_g}U_\Lambda}^2\dd S_g\dd s.
 \label{eq:local-product}
\end{equation}
The constant is independent of the geometry of $F$.
\end{corollary}

\begin{proof}
Cover $\Gamma_0\times I_*$ by finitely many boundary space--time cylinders whose doubled cylinders remain in $\Gamma_1\times I_s$ and whose local $W^{2,\infty}$ and metric constants are uniformly controlled. Since $\abs{F\times I_*}\ge\eta\abs{I_*}$, one cylinder contains a portion of $F\times I_*$ of measure at least a fixed number $m_\eta>0$. The low-value selection retains at least $m_\eta/2$, so all constants in the measurable-boundary continuation estimate are uniform. On each candidate cylinder the proof of Proposition~\ref{prop:local-spectral} produces a fixed nonempty interior set $\omega\Subset\M^\circ$; the corresponding instance of \eqref{eq:local-LR} supplies the required spectral concentration. Taking the maximum of the finitely many constants proves the result.
\end{proof}

The global manifold proof of Section~\ref{sec:heat} used a global $H^2$ regularity statement and a global conormal-trace estimate. Neither statement is needed here. Only the conormal trace on the observed regular patch has to be controlled.

\begin{lemma}
\label{lem:local-normal-trace}
Let $(\M,g)$ be a compact connected Lipschitz Riemannian manifold with uniformly elliptic Lipschitz metric, and let $\Gamma_0\Subset\Gamma_1\subset\partial\M$. Assume that the boundary is $W^{2,\infty}$ near $\Gamma_1$. Then there exists $C>0$ such that
\begin{equation}
 \norm{\partial_{\nu_g}v}_{L^2(\Gamma_0,dS_g)}
 \le C\norm{A_Dv}_{L^2(\M)},
 \qquad v\in D(A_D).
 \label{eq:local-normal-trace}
\end{equation}
\end{lemma}

\begin{proof}
Choose boundary neighborhoods $U_0,U_1$ of $\Gamma_0,\Gamma_1$ such that $U_0\Subset U_1$ and the boundary is $W^{2,\infty}$ throughout $U_1$. If $v\in D(A_D)$, then $v\in H_0^1(\M)$ and, in the weak sense,
\[
 -\Delta_gv=A_Dv\in L^2(\M).
\]
In every boundary chart contained in $U_1$, this equation has divergence form with coefficient matrix $\sqrt{\abs g}\,g^{ij}\in W^{1,\infty}$ and an $L^2$ right-hand side. Standard local boundary $H^2$ regularity on a $W^{2,\infty}$ graph patch therefore gives, after summing over a finite cover,
\begin{equation}
 \norm v_{H^2(U_0\cap\M)}
 \le C\left(
  \norm{A_Dv}_{L^2(U_1\cap\M)}
  +\norm v_{H^1(\M)}
 \right),
 \label{eq:local-H2}
\end{equation}
where the local Sobolev norm is understood in the fixed boundary charts. Equivalently, one may flatten the boundary with the metric-normal coordinates of Lemma~\ref{lem:cap}, apply the Dirichlet odd reflection, and use the interior $H^2$ estimate for the reflected equation with Lipschitz coefficients and an $L^2$ right-hand side. No boundary regularity outside $U_1$ enters the argument.

The local trace theorem, the Lipschitz regularity of the metric normal, and \eqref{eq:local-H2} imply
\begin{equation}
 \norm{\partial_{\nu_g}v}_{L^2(\Gamma_0,dS_g)}
 \le C\left(
  \norm{A_Dv}_{L^2(\M)}
  +\norm v_{H^1(\M)}
 \right).
 \label{eq:local-trace-intermediate}
\end{equation}
Since $\M$ is compact, connected, and has nonempty Dirichlet boundary, the first Dirichlet eigenvalue is positive. The spectral theorem gives
\[
 \norm v_{L^2(\M)}
 \le\lambda_1^{-1}\norm{A_Dv}_{L^2(\M)}.
\]
Moreover, the Dirichlet form identity yields
\[
 \norm{\nabla_g v}_{L^2(\M)}^2
 =\ip{A_Dv}{v}_{L^2(\M)}
 \le\norm{A_Dv}_{L^2(\M)}\norm v_{L^2(\M)}
 \le\lambda_1^{-1}\norm{A_Dv}_{L^2(\M)}^2.
\]
Thus $\norm v_{H^1(\M)}\le C\norm{A_Dv}_{L^2(\M)}$, and \eqref{eq:local-normal-trace} follows from \eqref{eq:local-trace-intermediate}.
\end{proof}

\begin{remark}
\label{rem:local-domain}
The point of Lemma~\ref{lem:local-normal-trace} is that one never needs global $H^2$ regularity for the full Dirichlet operator domain, which may fail under merely Lipschitz boundary regularity. The heat proof observes only on $\Gamma_0$, and the local $W^{2,\infty}$ regularity near $\Gamma_1$ is exactly what is needed to define and estimate that local conormal trace.
\end{remark}

Let $u(t)=e^{-tA_D}u_0$. The arguments of Section~\ref{sec:heat} remain valid after two replacements: Corollary~\ref{cor:local-product} is used in place of the global product-set spectral inequality, and Lemma~\ref{lem:local-normal-trace} replaces the global conormal-trace estimate.

\begin{proposition}
\label{prop:local-one-step}
Assume \eqref{eq:local-LR} and the local $W^{2,\infty}$ hypotheses on $\Gamma_0\Subset\Gamma_1$. For every $\kappa>0$, there are $C_\kappa>0$ and $\gamma_\kappa\in(0,1)$ such that, whenever $0<b-a\le1$ and $K\subset\Gamma_0\times(a,b)$ is measurable with
\[
 \abs K\ge\kappa S_g(\Gamma_0)(b-a),
\]
one has
\begin{equation}
 \norm{u(b)}_{L^2(\M)}
 \le C_\kappa\exp\left(\frac{C_\kappa}{b-a}\right)
 \norm{\partial_{\nu_g}u}_{L^2(K)}^{\gamma_\kappa}
 \norm{u(a)}_{L^2(\M)}^{1-\gamma_\kappa}.
 \label{eq:local-one-step}
\end{equation}
\end{proposition}

\begin{proof}
Set $L=b-a$. Apply Lemma~\ref{lem:time-window} with $\Gamma=\Gamma_0$. It gives an interior time window, a spatial set $F\subset\Gamma_0$ with $S_g(F)\ge c_\kappa S_g(\Gamma_0)$, and an internal observation set $K_f\subset K$ whose time fibers have a uniform positive relative measure.

Define the low-frequency boundary heat trace on the regular patch by
\[
 g_\Lambda(x,t)=\partial_{\nu_g}P_\Lambda^Du(x,t),
 \qquad x\in\Gamma_0.
\]
By Lemma~\ref{lem:local-normal-trace}, for every $k\ge0$,
\[
 \norm{\partial_t^kg_\Lambda(\cdot,\tau)}_{L^2(\Gamma_0,dS_g)}
 \le C\norm{A_D^{k+1}e^{-(\tau-a)A_D}P_\Lambda^Du(a)}_{L^2(\M)}.
\]
The same spectral calculation as in the global proof gives the localized analytic majorant
\[
 \norm{M_\Lambda}_{L^2(\Gamma_0,dS_g)}
 \le\frac{C_\kappa}{L}\norm{u(a)}_{L^2(\M)}.
\]
The recovery of the full time jet on the measurable fibers and the identity of Lemma~\ref{lem:time-jet-identity} are unchanged. Hence
\begin{equation}
 \norm{\partial_{\nu_g}V_{\Lambda,\tau}}_{L^2(F\times I_*)}
 \le C_\kappa e^{C_\kappa/L}
 \norm{g_\Lambda}_{L^2(K_f)}^{\theta_\kappa}
 \norm{u(a)}_{L^2(\M)}^{1-\theta_\kappa}.
 \label{eq:local-transfer}
\end{equation}
Apply Corollary~\ref{cor:local-product} to $V_{\Lambda,\tau}$. Since its coefficients are $f_je^{-(\tau-a)\lambda_j}$,
\[
 \norm{P_\Lambda^Du(\tau)}_{L^2(\M)}
 \le C_\kappa e^{C_\kappa/L+C_\kappa\sqrt\Lambda}
 \norm{g_\Lambda}_{L^2(K_f)}^{\theta_\kappa}
 \norm{u(a)}_{L^2(\M)}^{1-\theta_\kappa}.
\]
The low-frequency part contracts between $\tau$ and $b$.

For the high-frequency boundary trace, every time in $K_f$ is at distance at least $c_\kappa L$ from $a$. Again by Lemma~\ref{lem:local-normal-trace},
\begin{align*}
 \norm{\partial_{\nu_g}Q_\Lambda u(t)}_{L^2(\Gamma_0,dS_g)}
 &\le C\norm{A_De^{-(t-a)A_D}Q_\Lambda u(a)}_{L^2(\M)}\\
 &\le C_\kappa L^{-1}e^{-c_\kappa L\Lambda}
 \norm{u(a)}_{L^2(\M)}.
\end{align*}
Integration over the interior time window gives
\[
 \norm{g_\Lambda}_{L^2(K_f)}
 \le\norm{\partial_{\nu_g}u}_{L^2(K_f)}
 +C_\kappa L^{-1/2}e^{-c_\kappa L\Lambda}\norm{u(a)}_{L^2(\M)}.
\]
The high-frequency final state satisfies
\[
 \norm{Q_\Lambda u(b)}_{L^2(\M)}
 \le e^{-L\Lambda}\norm{u(a)}_{L^2(\M)}.
\]
Thus the same bridge estimate as in the global proof follows, with observation restricted to $K_f$. Frequency optimization via Lemma~\ref{lem:frequency-optimization} gives
\[
 \norm{u(b)}_{L^2(\M)}
 \le C_\kappa e^{C_\kappa/L}
 \norm{\partial_{\nu_g}u}_{L^2(K_f)}^{\gamma_\kappa}
 \norm{u(a)}_{L^2(\M)}^{1-\gamma_\kappa}.
\]
Since $K_f\subset K$, this implies \eqref{eq:local-one-step}.
\end{proof}

\begin{proof}[Proof of Theorem~\ref{thm:localized}]
Let $J\subset\Gamma_0\times(0,T)$ be measurable with $\abs J>0$. Repeat the density-point construction of Section~\ref{sec:heat} with the fixed spatial patch $\Gamma_0$. Fubini's theorem produces a positive-measure set of times for which the spatial fibers of $J$ have uniform positive $dS_g$-measure. At a suitable density point, choose the same geometric sequence of time intervals and the corresponding pairwise disjoint blocks $K_m\subset J$.

By Proposition~\ref{prop:local-one-step}, the local interpolation estimate on every interval has constants independent of $m$. The weighted Young inequality, the choice $w_m=e^{-K_0/L_m}$, and the telescoping identity are measure theoretic and unchanged. Summing the one-step inequalities gives
\[
 \norm{u(\ell_1)}_{L^2(\M)}
 \le C_{J,T}\norm{\partial_{\nu_g}u}_{L^2(J)}.
\]
Since $\ell_1<T$ and the Dirichlet heat semigroup is contractive,
\[
 \norm{u(T)}_{L^2(\M)}
 \le\norm{u(\ell_1)}_{L^2(\M)}.
\]
Squaring proves \eqref{eq:localized-main}.
\end{proof}

\begin{corollary}
\label{cor:star-shaped}
Let $\Omega\subset\R^n$ be a bounded connected Lipschitz and locally star-shaped domain. Let $\Gamma_0\Subset\Gamma_1\subset\partial\Omega$, and assume that $\partial\Omega$ is $W^{2,\infty}$ near $\Gamma_1$. Then the conclusion of Theorem~\ref{thm:localized} holds for every measurable $J\subset\Gamma_0\times(0,T)$ of positive measure, with $g$ equal to the Euclidean metric.
\end{corollary}

\begin{proof}
Apraiz--Escauriaza--Wang--Zhang~\cite[Theorem~3]{ApraizEtAl} prove that every bounded Lipschitz and locally star-shaped domain satisfies their interior-ball spectral estimate \cite[(1.4)]{ApraizEtAl}. As observed after \eqref{eq:local-LR}, that estimate implies the intrinsic property \eqref{eq:local-LR}. Apply Theorem~\ref{thm:localized} with the Euclidean metric.
\end{proof}

\begin{corollary}
\label{cor:almost-everywhere-C11}
Assume that $(\M,g)$ is a compact connected Lipschitz Riemannian manifold satisfying \eqref{eq:local-LR}. Let
\[
 \mathcal R_{1,1}=\{x\in\partial\M:\partial\M\text{ is } W^{2,\infty}\text{ in a neighborhood of }x\}.
\]
If $S_g(\partial\M\setminus\mathcal R_{1,1})=0$, then every measurable $J\subset\partial\M\times(0,T)$ with $\abs J>0$ satisfies
\[
 \norm{u(T)}_{L^2(\M)}^2
 \le C_{\M,g,T,J}\int_J\abs{\partial_{\nu_g}u}^2\dd S_g\dd t.
\]
\end{corollary}

\begin{proof}
Since the exceptional spatial set has surface measure zero,
\[
 \abs{J\cap(\mathcal R_{1,1}\times(0,T))}=\abs J>0.
\]
The relatively open set $\mathcal R_{1,1}$ is second countable and can be covered by countably many pairs $\Gamma_{0,k}\Subset\Gamma_{1,k}$ on which the boundary is $W^{2,\infty}$. At least one intersection $J\cap(\Gamma_{0,k}\times(0,T))$ has positive measure. Apply Theorem~\ref{thm:localized} to that intersection and use its inclusion in $J$.
\end{proof}

\begin{remark}
\label{rem:AEWZ-relation}
In the Euclidean setting, the uniform interior-ball spectral hypothesis \cite[(1.4)]{ApraizEtAl} is a concrete sufficient condition for the intrinsic spectral concentration assumption \eqref{eq:local-LR}. Their measurable-boundary theorem \cite[Theorem~2 and Remark~11]{ApraizEtAl} further assumes local real analyticity of the observed boundary patch; the cited remark permits piecewise analytic patches or a surface-null exceptional set. The present localized manifold theorem instead uses local $W^{2,\infty}$ regularity and an $L^2$ observation. Thus the geometric and spectral hypotheses in Theorem~\ref{thm:localized} are naturally modeled on, and extend to the manifold setting, the assumptions used in the AEWZ framework, while their $L^1$ observation remains stronger whenever both results apply.
\end{remark}

\begin{remark}
\label{rem:global-formulation}
The global $W^{2,\infty}$ manifold framework of Sections~\ref{sec:dirichlet}--\ref{sec:heat} treats the Dirichlet and Neumann spectral inequalities simultaneously and permits the Burq--Moyano concentration theorem to provide the global spectral input automatically. The present section separates that global spectral input from boundary regularity: for the Dirichlet heat equation, once \eqref{eq:local-LR} is available, $W^{2,\infty}$ regularity is needed only near the observed boundary patch, even when the remaining boundary is merely Lipschitz.
\end{remark}

\section{Proof of Theorem~\ref{thm:inverse-unique}: the initial-state inverse problem}
\label{sec:inverse}

We first record the semigroup representation needed for the positive-time solution class used above.

\begin{lemma}
\label{lem:homogeneous-semigroup}
Let
\[
 w\in C([0,T];L^2(\M))\cap H^1_{\mathrm{loc}}((0,T];L^2(\M))
 \cap L^2_{\mathrm{loc}}((0,T];H^2(\M))
\]
solve
\[
 \partial_tw-\Delta_gw=0\quad\text{in }\M\times(0,T),
 \qquad
 w|_{\partial\M\times(0,T)}=0
\]
in the positive-time trace sense. If $h=w(\cdot,0)$, then
\[
 w(t)=e^{-tA_D}h,
 \qquad 0\le t\le T.
\]
\end{lemma}

\begin{proof}
Fix $0<\varepsilon<t\le T$. On $(\varepsilon,t)$ one has
\[
 w\in H^1((\varepsilon,t);L^2(\M))
 \cap L^2((\varepsilon,t);D(A_D)),
 \qquad
 \partial_tw+A_Dw=0.
\]
The standard energy uniqueness argument for this strong problem gives
\[
 w(t)=e^{-(t-\varepsilon)A_D}w(\varepsilon).
\]
Indeed, the difference between the two sides is a homogeneous strong solution with zero value at $\varepsilon$, and differentiating its squared $L^2$ norm gives the nonpositive quantity $-2\norm{A_D^{1/2}z}^2$.
Since $w(\varepsilon)\to h$ in $L^2(\M)$ as $\varepsilon\downarrow0$, contractivity and strong continuity of the semigroup yield
\begin{align*}
 \norm{e^{-(t-\varepsilon)A_D}w(\varepsilon)-e^{-tA_D}h}_{L^2(\M)}
 &\le \norm{w(\varepsilon)-h}_{L^2(\M)}\\
 &\quad+\norm{(e^{-(t-\varepsilon)A_D}-e^{-tA_D})h}_{L^2(\M)}
 \longrightarrow0.
\end{align*}
Thus $w(t)=e^{-tA_D}h$ for $t>0$, and the identity at $t=0$ follows by continuity.
\end{proof}

We prove the unique identifiability theorem. The central point is that equality of the full lateral Dirichlet traces converts the difference of two admissible solutions into a homogeneous Dirichlet heat solution, to which Theorem~\ref{thm:heat-observability} applies.
\begin{proof}[Proof of Theorem~\ref{thm:inverse-unique}]
Let $u_1,u_2\in\mathscr H_T$ satisfy
\[
 \mathcal M_J^T(u_1)=\mathcal M_J^T(u_2).
\]
The first components of the datasets give
\[
 u_1=u_2
 \quad\text{on }\partial\M\times(0,T).
\]
Set
\[
 w=u_1-u_2,
 \qquad
 h=u_1(\cdot,0)-u_2(\cdot,0).
\]
Then
\begin{equation}
 \begin{cases}
  \partial_tw-\Delta_gw=0&\text{in }\M\times(0,T),\\
  w=0&\text{on }\partial\M\times(0,T),\\
  w(\cdot,0)=h&\text{in }\M.
 \end{cases}
 \label{eq:difference-homogeneous}
\end{equation}
By Lemma~\ref{lem:homogeneous-semigroup},
\[
 w(t)=e^{-tA_D}h.
\]
The second components of the datasets imply
\[
 \partial_{\nu_g}w=0
 \quad\text{a.e. on }J.
\]
By Theorem~\ref{thm:heat-observability},
\[
 \norm{e^{-TA_D}h}_{L^2(\M)}^2
 \le C\int_J\abs{\partial_{\nu_g}w}^2\dd S_g\dd t
 =0.
\]
Write
\[
 h=\sum_{j\ge1}h_j\phi_j.
\]
Then
\[
 0=e^{-TA_D}h
 =\sum_{j\ge1}e^{-T\lambda_j}h_j\phi_j.
\]
Since every multiplier $e^{-T\lambda_j}$ is strictly positive, orthogonality gives $h_j=0$ for all $j$, and hence
\[
 u_1(\cdot,0)=u_2(\cdot,0).
\]
Conversely, if the initial states and full Dirichlet traces agree, uniqueness of the forward initial--boundary value problem gives $u_1=u_2$, hence equality of the datasets. Applying Theorem~\ref{thm:heat-observability} to the same difference, without assuming that its conormal trace vanishes, proves \eqref{eq:terminal-stability}.

We next record the bounded positive-time observation operator and the stable terminal-state map, which will be used in the regularization argument. Let $J\subset\partial\M\times(0,T)$ have positive measure. Choose an integer $m_*\ge1$ such that $1/m_*<T$. Since
\[
 J=\bigcup_{m\ge m_*}
 J\cap\bigl(\partial\M\times(1/m,T)\bigr)
\]
up to the null slice at $t=0$, at least one member of this countable union has positive measure. Thus there exists $\tau\in(0,T)$ such that
\[
 J_\tau=J\cap\bigl(\partial\M\times(\tau,T)\bigr)
\]
has positive measure.

For $f\in L^2(\M)$ define
\begin{equation}
 \mathcal O_{J_\tau}f
 =\left.\partial_{\nu_g}e^{-tA_D}f\right|_{J_\tau}.
 \label{eq:observation-operator}
\end{equation}
This map is bounded. Indeed, by Lemma~\ref{lem:normal-trace},
\begin{align}
 \norm{\mathcal O_{J_\tau}f}_{L^2(J_\tau)}^2
 &\le C\int_\tau^T\norm{A_De^{-tA_D}f}_{L^2(\M)}^2\dd t
 \notag\\
 &=C\sum_{j=1}^\infty\lambda_j^2\abs{f_j}^2
 \int_\tau^Te^{-2t\lambda_j}\dd t
 \notag\\
 &\le\frac C2\sum_{j=1}^\infty
 \lambda_je^{-2\tau\lambda_j}\abs{f_j}^2
 \notag\\
 &\le\frac{C}{4e\tau}\norm f_{L^2(\M)}^2,
 \label{eq:observation-bounded}
\end{align}
where $f_j=(f,\phi_j)$ and
\[
 \sup_{\lambda\ge0}\lambda e^{-2\tau\lambda}=(2e\tau)^{-1}.
\]

Let
\[
 \mathcal Y_\tau=\Ran(\mathcal O_{J_\tau})\subset L^2(J_\tau).
\]
For $y=\mathcal O_{J_\tau}f$, define
\begin{equation}
 \mathcal R_Ty=e^{-TA_D}f.
 \label{eq:terminal-map}
\end{equation}
This is well defined. If
\[
 \mathcal O_{J_\tau}f_1=\mathcal O_{J_\tau}f_2,
\]
then Theorem~\ref{thm:heat-observability}, applied on $J_\tau$, gives
\[
 \norm{e^{-TA_D}(f_1-f_2)}_{L^2(\M)}=0.
\]
Therefore $e^{-TA_D}f_1=e^{-TA_D}f_2$. The same observability estimate gives
\begin{equation}
 \norm{\mathcal R_Ty}_{L^2(\M)}
 \le C_{\mathrm{obs}}\norm y_{L^2(J_\tau)},
 \qquad y\in\mathcal Y_\tau.
 \label{eq:terminal-recovery}
\end{equation}
Thus the terminal state is recovered Lipschitz stably on the data range.

This argument also explains why a naive time reversal is unnecessary. On the range of the heat semigroup,
\[
 f=e^{TA_D}u(T)
 =\sum_{j=1}^\infty e^{T\lambda_j}(u(T),\phi_j)\phi_j,
\]
but $e^{TA_D}$ is unbounded. Injectivity is enough for exact uniqueness; stable recovery of the initial state requires an a priori spectral smoothness bound. This completes the proof of Theorem~\ref{thm:inverse-unique}.
\end{proof}
\begin{remark}
\label{rem:product-observation}
If $\Gamma\subset\partial\M$ and $E\subset(0,T)$ have positive measures, then $J=\Gamma\times E$ has positive surface--time measure. Therefore the full lateral Dirichlet trace together with the conormal derivative on $\Gamma\times E$ uniquely determines the initial state.
\end{remark}

\section{Proof of Theorem~\ref{thm:inverse-log}: conditional stability, regularization, and sharpness}
\label{sec:stability}

Throughout this section, $A=A_D$. Let $u_1,u_2\in\mathscr H_T$ have the same full lateral Dirichlet trace, and set
\[
 h=u_1(\cdot,0)-u_2(\cdot,0),
 \qquad
 w=u_1-u_2.
\]
Then $w$ solves the homogeneous Dirichlet problem and Lemma~\ref{lem:homogeneous-semigroup} gives
\[
 w(t)=e^{-tA}h.
\]
Thus every argument below is performed on an affine class of solutions with fixed full Dirichlet trace. The homogeneous Dirichlet problem is the special class with zero trace.

\subsection{The logarithmic estimate}
\begin{proof}[Proof of Theorem~\ref{thm:inverse-log}]
Let $h\in D(A^{s/2})$ and fix $\mathfrak M\ge0$ such that
\[
 \norm{A^{s/2}h}_{L^2(\M)}\le\mathfrak M.
\]
Set
\[
 \varepsilon=\norm{e^{-TA}h}_{L^2(\M)}.
\]
We first prove the purely spectral backward-heat inequality
\begin{equation}
 \norm h_{L^2(\M)}
 \le C\mathfrak M
 \left[
  \log\left(e+\frac{\mathfrak M}{\varepsilon}\right)
 \right]^{-s/2}.
 \label{eq:backward-log}
\end{equation}
If $\mathfrak M=0$, positivity of $\lambda_1$ gives $h=0$. If $\varepsilon=0$, injectivity of the heat semigroup gives $h=0$. Assume $\mathfrak M,\varepsilon>0$ and write
\[
 h=\sum_{j\ge1}h_j\phi_j.
\]
For $R\ge\lambda_1$, let
\[
 P_R=\one_{[0,R]}(A),
 \qquad
 Q_R=\Id-P_R.
\]
The low-frequency part satisfies
\begin{align}
 \norm{P_Rh}_{L^2(\M)}^2
 &\le e^{2TR}
 \sum_{\lambda_j\le R}e^{-2T\lambda_j}\abs{h_j}^2
 \notag\\
 &\le e^{2TR}\varepsilon^2,
 \label{eq:backward-low}
\end{align}
while the a priori bound controls the high-frequency tail:
\begin{align}
 \norm{Q_Rh}_{L^2(\M)}^2
 &\le R^{-s}\sum_{\lambda_j>R}\lambda_j^s\abs{h_j}^2
 \notag\\
 &\le R^{-s}\mathfrak M^2.
 \label{eq:backward-high}
\end{align}
Hence
\begin{equation}
 \norm h_{L^2(\M)}^2
 \le e^{2TR}\varepsilon^2+R^{-s}\mathfrak M^2.
 \label{eq:backward-split}
\end{equation}
Put
\[
 L=\log\left(e+\frac{\mathfrak M}{\varepsilon}\right)\ge1.
\]
If $L$ is bounded by a fixed number depending on $T$ and $\lambda_1$, then
\[
 \norm h_{L^2(\M)}
 \le\lambda_1^{-s/2}\mathfrak M
\]
already has the form \eqref{eq:backward-log}. Otherwise choose
\[
 R=\frac{L}{2T}\ge\lambda_1.
\]
For $L\ge2$,
\[
 \frac{\mathfrak M}{\varepsilon}=e^L-e\ge\frac12e^L,
 \qquad
 \varepsilon\le2\mathfrak Me^{-L}.
\]
Inserting this into \eqref{eq:backward-split},
\[
 \norm h_{L^2(\M)}^2
 \le4\mathfrak M^2e^{-L}
 +(2T)^s\mathfrak M^2L^{-s}
 \le C\mathfrak M^2L^{-s}.
\]
Taking square roots proves \eqref{eq:backward-log}.

For the difference $w=u_1-u_2=e^{-tA}h$, set
\[
 \delta
 =\norm{\partial_{\nu_g}w}_{L^2(J_\tau)}
 =\norm{\partial_{\nu_g}(u_1-u_2)}_{L^2(J_\tau)}.
\]
The observability estimate gives
\[
 \varepsilon=\norm{e^{-TA}h}_{L^2(\M)}
 \le C_{\mathrm{obs}}\delta.
\]
Combining this with \eqref{eq:backward-log} yields \eqref{eq:main-log}, with $c=C_{\mathrm{obs}}^{-1}$ after a harmless adjustment of the outside constant.
\end{proof}
If the individual initial states satisfy
\[
 \norm{A^{s/2}f_i}_{L^2(\M)}\le\mathfrak M_0,
 \qquad i=1,2,
\]
then
\[
 \norm{A^{s/2}(f_1-f_2)}_{L^2(\M)}\le2\mathfrak M_0,
\]
so the same estimate holds with $\mathfrak M_0$ after renaming constants.

The estimate also gives a diameter bound for noisy-data admissible sets within a fixed full Dirichlet-trace class. Suppose $u_1,u_2\in\mathscr H_T$ have the same full lateral Dirichlet trace and that noisy conormal data $y^\delta\in L^2(J_\tau)$ satisfy
\[
 \norm{\partial_{\nu_g}u_i|_{J_\tau}-y^\delta}_{L^2(J_\tau)}\le\delta,
 \qquad i=1,2.
\]
Since the common Dirichlet trace cancels in the difference and Lemma~\ref{lem:homogeneous-semigroup} gives $u_1-u_2=e^{-tA}(f_1-f_2)$,
\[
 \norm{\mathcal O_{J_\tau}(f_1-f_2)}_{L^2(J_\tau)}
 =\norm{\partial_{\nu_g}(u_1-u_2)}_{L^2(J_\tau)}
 \le2\delta.
\]
Hence the set of all initial states in the a priori ball, within the prescribed full Dirichlet-trace class, which fit the partial conormal data to accuracy $\delta$ has $L^2$ diameter bounded by
\begin{equation}
 C\mathfrak M_0
 \left[
  \log\left(e+\frac{c\mathfrak M_0}{\delta}\right)
 \right]^{-s/2}.
 \label{eq:noisy-diameter}
\end{equation}

\subsection{A spectral-cutoff reconstruction on a fixed
Dirichlet-trace class}

Throughout this subsection, we write $A=A_D$. Recall that the
positive-time observation operator
\[
O_{J_\tau}:L^2(M)\longrightarrow L^2(J_\tau),
\qquad
O_{J_\tau}h
=
\left.\partial_{\nu_g}e^{-tA}h\right|_{J_\tau},
\]
is bounded. Set
\[
Y_\tau^0:=\operatorname{Ran}(O_{J_\tau})
    \subset L^2(J_\tau),
\qquad
\overline Y_\tau
:=
\overline{Y_\tau^0}^{\,L^2(J_\tau)}.
\]
We do not assume that $Y_\tau^0$ is closed.

For $y=O_{J_\tau}h\in Y_\tau^0$, define
\[
R_Ty:=e^{-TA}h.
\]
This definition is independent of the choice of $h$. Indeed, if
\[
O_{J_\tau}h_1=O_{J_\tau}h_2,
\]
then Theorem~1.3, applied to $e^{-tA}(h_1-h_2)$ on $J_\tau$,
gives
\[
e^{-TA}(h_1-h_2)=0.
\]
Moreover, the observability estimate~(6.5) yields
\[
\|R_Ty\|_{L^2(M)}
\le
C_{\mathrm{obs}}\|y\|_{L^2(J_\tau)},
\qquad y\in Y_\tau^0.
\]
Thus $R_T:Y_\tau^0\to L^2(M)$ is bounded and admits a unique
bounded extension
\[
\overline R_T:\overline Y_\tau\longrightarrow L^2(M),
\qquad
\|\overline R_T\|
\le C_{\mathrm{obs}}.
\]

Let
\[
\Pi_\tau:L^2(J_\tau)\longrightarrow\overline Y_\tau
\]
be the orthogonal projection, and define
\[
\widetilde R_T:=\overline R_T\Pi_\tau.
\]
Since $\|\Pi_\tau\|=1$, we have
\begin{equation}
\label{eq:extended-terminal-recovery}
\|\widetilde R_T\|
\le C_{\mathrm{obs}},
\qquad
\widetilde R_TO_{J_\tau}h=e^{-TA}h
\quad\text{for every }h\in L^2(M).
\end{equation}
Indeed, $O_{J_\tau}h\in Y_\tau^0\subset\overline Y_\tau$, and
hence
\[
\Pi_\tau O_{J_\tau}h=O_{J_\tau}h.
\]

For $R\ge\lambda_1$, let
\[
P_R=\mathbf 1_{[0,R]}(A),
\qquad
Q_R=I-P_R,
\]
and define the bounded spectral back-propagation operator
\begin{equation}
\label{eq:spectral-back-propagation}
B_Rz
:=
\sum_{\lambda_j\le R}
e^{T\lambda_j}
(z,\phi_j)_{L^2(M)}\phi_j,
\qquad z\in L^2(M).
\end{equation}
Then
\[
\|B_R\|\le e^{TR},
\qquad
B_Re^{-TA}h=P_Rh.
\]

We now pass from the homogeneous reconstruction to a fixed
full-Dirichlet-trace class. Fix a full lateral Dirichlet trace
$b$ and a reference solution $u^\circ\in\mathcal H_T$ having
that trace. Set
\[
f^\circ:=u^\circ(\cdot,0).
\]
For any other $u\in\mathcal H_T$ with the same full lateral
Dirichlet trace, put
\[
f:=u(\cdot,0),
\qquad
h:=f-f^\circ,
\qquad
z:=u-u^\circ.
\]
By Lemma~6.1, the difference $z$ satisfies the homogeneous
Dirichlet problem and hence
\[
z(t)=e^{-tA}h.
\]
Therefore, the residual partial conormal datum is
\begin{equation}
\label{eq:exact-residual-datum}
y
:=
\left.\partial_{\nu_g}(u-u^\circ)\right|_{J_\tau}
=
O_{J_\tau}h.
\end{equation}

Assume that, for some $s>0$,
\[
h\in D(A^{s/2}),
\qquad
\|A^{s/2}h\|_{L^2(M)}
\le\mathfrak M.
\]
Let $m^\delta\in L^2(J_\tau)$ be a noisy measurement of
$\left.\partial_{\nu_g}u\right|_{J_\tau}$ satisfying
\[
\left\|
m^\delta-\left.\partial_{\nu_g}u\right|_{J_\tau}
\right\|_{L^2(J_\tau)}
\le\delta.
\]
Since the reference solution is known, define the residual noisy
datum by
\[
y^\delta
:=
m^\delta-\left.\partial_{\nu_g}u^\circ\right|_{J_\tau}.
\]
It follows from \eqref{eq:exact-residual-datum} that
\[
\|y^\delta-y\|_{L^2(J_\tau)}
\le\delta.
\]

For every $R\ge\lambda_1$, define
\[
h_R^\delta:=B_R\widetilde R_Ty^\delta,
\qquad
f_R^\delta:=f^\circ+h_R^\delta.
\]
By \eqref{eq:extended-terminal-recovery},
\[
B_R\widetilde R_Ty
=
B_Re^{-TA}h
=
P_Rh.
\]
Consequently,
\[
\begin{aligned}
h-h_R^\delta
&=
Q_Rh+P_Rh-B_R\widetilde R_Ty^\delta \\
&=
Q_Rh+B_R\widetilde R_T(y-y^\delta).
\end{aligned}
\]
Using the a priori bound, the spectral theorem, and
\eqref{eq:extended-terminal-recovery}, we obtain
\begin{equation}
\label{eq:spectral-cutoff-error}
\begin{aligned}
\|f-f_R^\delta\|_{L^2(M)}
&=
\|h-h_R^\delta\|_{L^2(M)} \\
&\le
\|Q_Rh\|_{L^2(M)}
+
\|B_R\|\,\|\widetilde R_T\|\,
\|y-y^\delta\|_{L^2(J_\tau)} \\
&\le
R^{-s/2}\mathfrak M
+
C_{\mathrm{obs}}e^{TR}\delta.
\end{aligned}
\end{equation}
The first term is the spectral truncation error, whereas the
second describes the exponential amplification of the residual
data noise.

We next choose the cutoff parameter for arbitrary noise levels.
Assume first that $\mathfrak M>0$ and $\delta>0$, and set
\[
q_\delta
:=
\frac{C_{\mathrm{obs}}\delta}{\mathfrak M},
\qquad
L_\delta
:=
\log\bigl(e+q_\delta^{-1}\bigr).
\]
Define
\begin{equation}
\label{eq:all-noise-reconstruction}
\widehat h^\delta
:=
\begin{cases}
B_{R_\delta}\widetilde R_Ty^\delta,
    & 0<q_\delta<1,\\[2mm]
0,
    & q_\delta\ge1,
\end{cases}
\qquad
\widehat f^\delta
:=
f^\circ+\widehat h^\delta,
\qquad
R_\delta
:=
\max\left\{
\lambda_1,\frac{L_\delta}{2T}
\right\}.
\end{equation}
If $\mathfrak M=0$, then $h=0$; in this case we set
\[
\widehat h^\delta=0,
\qquad
\widehat f^\delta=f^\circ.
\]

Suppose that $0<q_\delta<1$. If
\[
\frac{L_\delta}{2T}\ge\lambda_1,
\]
then $R_\delta=L_\delta/(2T)$, and
\[
R_\delta^{-s/2}\mathfrak M
\le
C\mathfrak M L_\delta^{-s/2}.
\]
Moreover,
\[
\begin{aligned}
C_{\mathrm{obs}}e^{TR_\delta}\delta
&=
\mathfrak M q_\delta
\bigl(e+q_\delta^{-1}\bigr)^{1/2} \\
&\le
C\mathfrak M q_\delta^{1/2}
\le
C_s\mathfrak M L_\delta^{-s/2}.
\end{aligned}
\]
Applying \eqref{eq:spectral-cutoff-error} with
$R=R_\delta$ therefore gives
\[
\|f-\widehat f^\delta\|_{L^2(M)}
\le
C\mathfrak M L_\delta^{-s/2}.
\]

If instead
\[
\frac{L_\delta}{2T}<\lambda_1,
\]
then $R_\delta=\lambda_1$. Since $q_\delta<1$, we have
$C_{\mathrm{obs}}\delta<\mathfrak M$, and hence
\[
\begin{aligned}
\|f-\widehat f^\delta\|_{L^2(M)}
&\le
\lambda_1^{-s/2}\mathfrak M
+
C_{\mathrm{obs}}e^{T\lambda_1}\delta \\
&\le
\bigl(\lambda_1^{-s/2}+e^{T\lambda_1}\bigr)
\mathfrak M.
\end{aligned}
\]
In this regime,
\[
1\le L_\delta<2T\lambda_1,
\]
so, after enlarging the constant,
\[
\|f-\widehat f^\delta\|_{L^2(M)}
\le
C\mathfrak M L_\delta^{-s/2}.
\]

Finally, suppose that $q_\delta\ge1$. By
\eqref{eq:all-noise-reconstruction} and the spectral gap,
\[
\|f-\widehat f^\delta\|_{L^2(M)}
=
\|h\|_{L^2(M)}
\le
\lambda_1^{-s/2}\mathfrak M.
\]
Since
\[
1\le L_\delta\le\log(e+1),
\]
the same logarithmic estimate follows after another enlargement
of the constant. We have therefore proved that, for every
$\delta>0$,
\begin{equation}
\label{eq:all-noise-logarithmic-rate}
\|f-\widehat f^\delta\|_{L^2(M)}
\le
C\mathfrak M
\left[
\log\left(
e+\frac{\mathfrak M}{C_{\mathrm{obs}}\delta}
\right)
\right]^{-s/2}.
\end{equation}

For exact residual data $y=O_{J_\tau}h$, define
\[
h_R^0:=B_R\widetilde R_Ty.
\]
Then
\[
h_R^0=P_Rh,
\qquad
f_R^0
:=
f^\circ+h_R^0
=
f^\circ+P_R(f-f^\circ).
\]
It follows that
\[
f_R^0\longrightarrow f
\quad\text{in }L^2(M)
\quad\text{as }R\to\infty.
\]
Thus the spectral cutoff defines a reconstruction for arbitrary
$L^2(J_\tau)$ noise on every fixed full-Dirichlet-trace class,
once a reference solution in that class has been fixed. In the
homogeneous Dirichlet case, one takes
\[
u^\circ=0,
\qquad
f^\circ=0.
\]
The resulting logarithmic convergence rate agrees with that of
Theorem~\ref{thm:inverse-log}.

\subsection{Instability and optimality of the logarithmic power}

The inverse observation map cannot be continuous on unrestricted $L^2$ data. Let
\[
 J_\tau\subset\partial\M\times(\tau,T)
\]
and take $f_j=\phi_j$. By Lemma~\ref{lem:normal-trace},
\[
 \norm{\partial_{\nu_g}\phi_j}_{L^2(\partial\M)}
 \le C\lambda_j.
\]
Therefore
\begin{align}
 \norm{\mathcal O_{J_\tau}\phi_j}_{L^2(J_\tau)}^2
 &\le C\lambda_j^2\int_\tau^Te^{-2t\lambda_j}\dd t
 \notag\\
 &\le C\lambda_je^{-2\tau\lambda_j}\longrightarrow0,
 \label{eq:unrestricted-instability}
\end{align}
whereas $\norm{\phi_j}_{L^2(\M)}=1$.

To prove the sharp statements in Theorem~\ref{thm:inverse-log}, set
\[
 h_j=\mathfrak M\lambda_j^{-s/2}\phi_j.
\]
Then
\[
 \norm{A^{s/2}h_j}_{L^2(\M)}=\mathfrak M,
 \qquad
 \norm{h_j}_{L^2(\M)}=\mathfrak M\lambda_j^{-s/2},
\]
and the preceding calculation gives
\begin{equation}
 \delta_j:=\norm{\mathcal O_{J_\tau}h_j}_{L^2(J_\tau)}
 \le C\mathfrak M\lambda_j^{(1-s)/2}e^{-\tau\lambda_j}.
 \label{eq:sharp-data}
\end{equation}
If a H\"older estimate
\[
 \norm h_{L^2(\M)}
 \le C\mathfrak M^{1-\alpha}
 \norm{\mathcal O_{J_\tau}h}_{L^2(J_\tau)}^\alpha
\]
held for some $\alpha>0$ on the a priori ball, then \eqref{eq:sharp-data} would imply
\[
 \mathfrak M\lambda_j^{-s/2}
 \le C'\mathfrak M\lambda_j^{\alpha(1-s)/2}e^{-\alpha\tau\lambda_j},
\]
which is impossible as $j\to\infty$.

Now suppose an estimate of logarithmic form held with an exponent $\beta>s/2$. From \eqref{eq:sharp-data}, for sufficiently large $j$,
\[
 \log\left(e+\frac{c\mathfrak M}{\delta_j}\right)
 \ge\tau\lambda_j-\frac{1-s}{2}\log\lambda_j-C
 \ge\frac\tau2\lambda_j.
\]
The assumed estimate would give
\[
 \mathfrak M\lambda_j^{-s/2}
 \le C'\mathfrak M\lambda_j^{-\beta},
\]
which is impossible because $\beta>s/2$. This proves the sharpness part of Theorem~\ref{thm:inverse-log}.

\begin{remark}
\label{rem:density}
All global manifold results remain valid for the weighted symmetric operator
\[
 A_\rho=-\rho^{-1}\operatorname{div}_g(\rho\nabla_g),
\]
where $\rho$ is positive and Lipschitz. One works in $L^2(\M,\rho\,dV_g)$ and uses the corresponding weighted conormal derivative. The quantitative boundary continuation theorem and the rough-manifold spectral concentration theorem are formulated at this level of generality. The metric-normal coordinate construction, the blockwise reflection, the time-jet transfer, and all spectral arguments are unchanged after the natural weighted substitutions.
\end{remark}

\begin{remark}
\label{rem:Neumann-heat}
The parabolic observability and inverse initial-state theorems proved in this paper concern the homogeneous Dirichlet problem, for which the natural observation is the conormal heat flux. The Neumann boundary spectral inequality of Theorem~\ref{thm:neu-spectral} indicates a parallel extension to the homogeneous Neumann heat equation
\[
 \partial_tv-\Delta_gv=0,
 \qquad
 \partial_{\nu_g}v=0
 \quad\text{on }\partial\M\times(0,T).
\]
In that setting the natural boundary observation is the temperature trace $v|_J$, since the conormal derivative is prescribed to vanish. To implement the argument one replaces the Dirichlet conormal-trace estimate by the boundary trace estimate
\[
 \norm w_{L^2(\partial\M)}
 \le C\norm{(\Id+A_N)^{1/2}w}_{L^2(\M)},
 \qquad A_N=-\Delta_{g,N},
\]
and then repeats the time-jet transfer, low--high-frequency decomposition, frequency optimization, and weighted telescoping construction with the Neumann spectral projectors. The factor $\Id+A_N$ is needed to retain the zero Neumann mode; on the mean-zero subspace it may be replaced by $A_N$. These substitutions also point to analogous uniqueness, conditional logarithmic stability, and spectral regularization statements for recovery of the Neumann initial state from measurable boundary-temperature data. We do not carry out those parallel details here, and no Neumann parabolic conclusion is used elsewhere in the paper.
\end{remark}

\end{document}